\documentclass[final,3p,times]{elsarticle}

\usepackage{amssymb}
\usepackage{amsthm}

\journal{}

\usepackage{showkeys}
\usepackage{amstext,amsmath,amsfonts,graphicx}
\usepackage{subcaption}

\newtheorem{assumption}{Assumption}
\newtheorem{theorem}{Theorem}[section]
\newtheorem{lemma}[theorem]{Lemma}
\newtheorem{corollary}[theorem]{Corollary}
\newtheorem{remark}{Remark}[section]

\newcommand{\triple}[1]{{\left\vert\kern-0.12ex\left\vert\kern-0.12ex\left\vert #1 
    \right\vert\kern-0.12ex\right\vert\kern-0.12ex\right\vert}}
\DeclareMathOperator{\Div}{div}
\newcommand{\Th}{\mathcal{T}_h}
\newcommand{\RR}{\mathbb{R}}
\newcommand{\ddn}[1]{\frac{\partial#1}{\partial n}}

\usepackage{color}
\newcommand{\red}[1]{#1}

\begin{document}

\begin{frontmatter}



\title{\red{CutFEM without cutting the mesh cells: a new way to impose Dirichlet and Neumann boundary conditions on unfitted meshes}}


\author{Alexei Lozinski}

\address{Laboratoire de Math\'ematiques de Besan\c{c}on - UMR CNRS 6623, \\ Univ. 
Bourgogne Franche Comt\'e, 16 route de Gray, 25030 Besan\c{c}on Cedex, France.}

\begin{abstract}
We present a method of CutFEM type for the Poisson 
problem with either Dirichlet or Neumann boundary conditions. The computational mesh is obtained from a background (typically 
uniform Cartesian) mesh by retaining only the elements intersecting the domain 
where the problem is posed. The resulting mesh does not thus fit the boundary 
of the problem domain. Several finite element methods (XFEM, CutFEM) adapted to 
such meshes have been recently proposed. The originality of the present article 
consists in avoiding integration over the elements cut by the boundary of the 
problem domain, while preserving the optimal convergence rates, as confirmed by 
both the theoretical estimates and the numerical results. 
\end{abstract}


\end{frontmatter}


\section{Introduction.}

In this article, we propose a new approach to the numerical solution of boundary value problems for partial differential equations using finite elements on non-matching meshes, circumventing the need to generate a mesh accurately fitting the physical boundaries or interfaces. Such approaches, classically known as the fictitious domain methods, have a long history dating back to \cite{saul63} in the case of the Poisson problem with Dirichlet boundary conditions. They were later popularized, by Glowinski and co-workers, cf. \cite{glowinski94} for example, and successfully applied in the context of particular flow simulations \cite{glowinski99, glowinski07}. The basic idea of the classical fictitious domain method is to embed the physical domain $\Omega$ into a bigger simply shaped domain $\mathcal{O}$, to extend the physically meaningful solution on $\Omega$ by a fictitious solution on $\mathcal{O}\setminus\Omega$ using the same governing equations as on $\Omega$, and to impose the boundary conditions on $\partial\Omega$ by Lagrange multipliers. At the numerical level, this means that one can work with simple meshes on $\mathcal{O}$, but one also needs a mesh on the physical boundary $\partial\Omega$ for the Lagrange multiplier, which should be coarser than the first one in order to satisfy the inf-sup condition \cite{girault95}. One is thus not completely free from meshing problems. Another unfortunate feature of such methods is their poor accuracy: one cannot expect the convergence order to be better than $1/2$ with the respect to the mesh size. Closely related penalty methods are well suited for both Dirichlet and Neumann boundary conditions \cite{glowinski92,maury09}, may be simpler to implement in practice than the fictitious domain methods with Lagrange multipliers, but share with them the poor convergence properties.  

More recently, several optimally convergent finite element methods on non-matching meshes were
proposed following the XFEM  or CutFEM  paradigms. 
XFEM (extended finite element method) was initially introduced in \cite{moes99} for applications in the structural mechanics on cracked domains (Neumann boundary conditions on the crack). Its ability to impose Dirichlet boundary conditions was demonstrated in \cite{moes06,sukumar2001} and a properly stabilized version with proved optimal convergence was proposed in  \cite{HaslingerRenard}. The CutFEM methods \cite{burman15} were first introduced in a series of papers by Burman and Hansbo {\cite{burman1,burman2,burman3}. The common feature of all these methods is that the simple background mesh is only used to define the finite element space (only the mesh elements having non empty intersection with $\Omega$ are kept in the computational mesh), but the solution is no longer extended from the physical domain $\Omega$ to a larger fictitious domain. The integrals over $\Omega$ are thus maintained in the finite element formulation. The boundary conditions are imposed either through Lagrange multipliers living on the same mesh as the primary solution {\cite{HaslingerRenard,burman1}} or by the Nitsche method {\cite{burman2,burman3}} stabilized by the ghost penalty \cite{burmanghost}. The optimal convergence, i.e. the error estimates of the same order as those for finite element methods on a comparable matching mesh, are established for all the methods above.

As already mentioned, XFEM/CutFEM methods contain the integrals over $\Omega$ in their formulations. This can be cumbersome in practice. Citing \cite{burman1} ``the only remaining difficulty of implementation is the actual integration on the boundary and on parts of elements cut by the boundary. This difficulty however is expected to arise in any optimal order fictitious domain method.'' We attempt in the present article to prove that this statement is not entirely true. We propose in fact an optimal order fictitious domain method that does not involve the integrals on $\Omega$ and thus does not require to perform the integration on parts of elements cut by the boundary.  Our method can be regarded as a mix between the classical fictitious domain approach and CutFEM. As in CutFEM, we use the the computational mesh constructed by keeping only the elements from the background mesh on the embedding domain $\mathcal{O}$ having non empty intersection with $\Omega$. On the other hand, we extend the solution from $\Omega$ to the computational domain, which is now only slightly larger than $\Omega$. In fact, the extension is done on a narrow band of width of order of the mesh size, contrary to the extension to entire $\mathcal{O}$ as in the classical fictitious domain approach.\footnote{The idea of constructing numerically a smooth extension to the whole $\mathcal{O}$ is explored in  {\cite{fabreges}} resulting in an optimally convergent method. The price to pay is the necessity to solve an optimization problem by an iterative process, which can be expensive in practice.} This minimizes the effect of choosing a ``wrong'' extension and enables us, with the help of a proper stabilization mainly borrowed from CutFEM, to preserve the optimal convergence without integration on the cut mesh elements. Unfortunately, the integration on the actual boundary should still be performed, so that some non trivial quadrature techniques are still needed. We believe however that an integration over the surface (resp. the curve) cut by the mesh is simpler to implement than that on the cut mesh elements in 3D (resp. 2D).

Let us now give a first sketch of the methods that we propose in this article. We restrict ourselves to the simple model problem: the Poisson equation with either Dirichlet boundary conditions
\begin{equation}
  - \Delta u = f \text{ in } \Omega, \qquad u = g \text{ on } \Gamma
  \label{Dir:P}
\end{equation}
or Neumann ones
\begin{equation}
  - \Delta u = f \text{ in } \Omega, \qquad \frac{\partial u}{\partial n}= g
  \text{ on } \Gamma \label{Neu:P}
\end{equation}
Here $\Omega \subset \mathbb{R}^d$, $d=2,3$ is a domain with smooth boundary $\Gamma=\partial\Omega$,
$f$ and $g$ are given functions on $\Omega$ and $\Gamma$ respectively (satisfying the compatibility condition $\int_\Omega f+\int_\Gamma g=0$ in the Neumann case). 

We start by embedding $\Omega$ into a simply shaped domain
$\mathcal{O}$ and introduce a quasi-uniform mesh $\mathcal{T}_h^{\mathcal{O}}$
on $\mathcal{O}$ consisting of triangles/tetrahedrons of maximum diameter $h$ that can be cut by the boundary $\Gamma$ in an arbitrary manner.
The computational mesh $\Th$ is obtained from $\mathcal{T}_h^{\mathcal{O}}$ by dropping out all the mesh element lying completely outside $\Omega$:
\begin{equation}\label{ThOmegah}
  \mathcal{T}_h = \{T \in \mathcal{T}_h^{\mathcal{O}} : T \cap \Omega \neq
  \varnothing\}, \qquad \Omega_h = (\cup_{T \in \mathcal{T}_h} T)^{\circ}
\end{equation}
as illustrated in Fig.~\ref{Dir:figdom}. $\Omega_h$ is thus the domain occupied by the computational mesh, slightly larger than $\Omega$. We denote its boundary by $\Gamma_h := \partial \Omega_h$.

\begin{figure}[htp]
\centerline{
\includegraphics[scale=0.28]{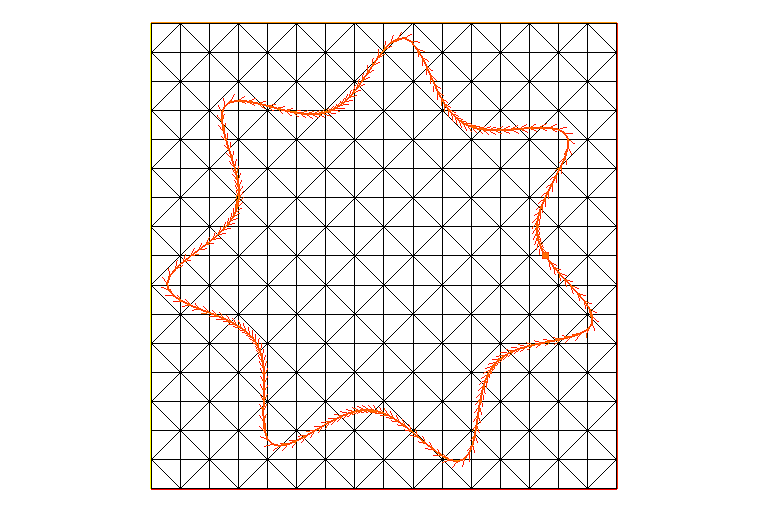}
\hspace{-15mm}
\includegraphics[scale=0.28]{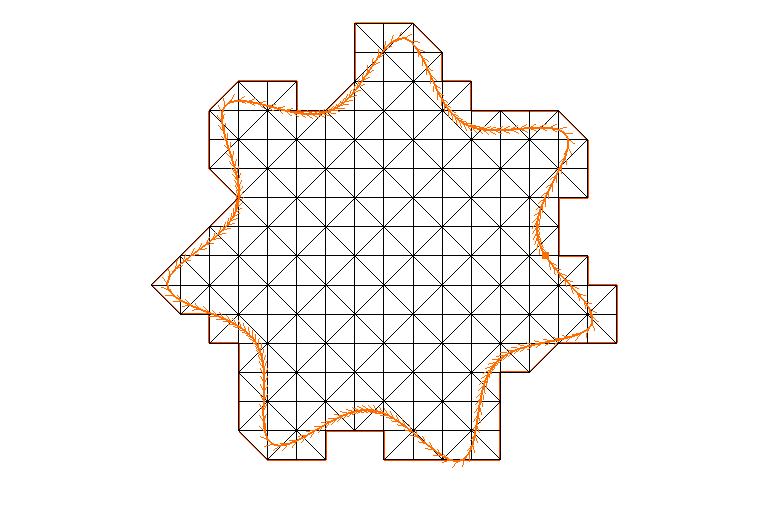}
\hspace{-15mm}
\includegraphics[scale=0.28]{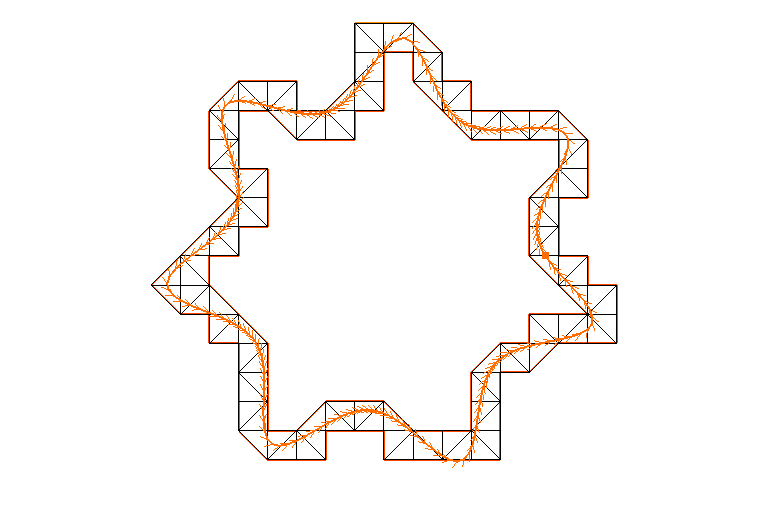}
}
\caption{Left: the ``physical'' domain $\Omega$ represented by its boundary $\Gamma$ embedded into a rectangle ${\mathcal{O}}$ with the ``background'' mesh $\mathcal{T}_{h}^{\mathcal{O}}$ on it. Center: the computational mesh $\Th$ obtained by dropping the unnecessary triangles from $\mathcal{T}_{h}^{\mathcal{O}}$. Right: the band $\Omega_h^\Gamma$ occupied by the cut elements $\Th^\Gamma$ of the mesh $\Th$.}
\label{Dir:figdom}
\end{figure}

We outline first the derivation of our method in the case of Dirichlet boundary conditions, following the preliminary version in {\cite{lozinski16}. We assume that the right-hand side $f$ is extended from $\Omega$ to $\Omega_h$ and imagine (for
the moment) that (\ref{Dir:P}) can be solved on the extended domain $\Omega_h$
while still imposing the boundary conditions on $\Gamma$:
\begin{equation}
  - \Delta u = f \text{ in } \Omega_h, \qquad u = g \text{ on } \Gamma .
  \label{Dir:Pext}
\end{equation}
We keep here the same notations $u$ and $f$ for the functions on $\Omega_h$ as
for the originals on $\Omega$. Integration by parts over $\Omega_h$ and
imposing the boundary conditions weakly on $\Gamma$ as in \red{the antisymmetric variant \cite{oden98} of the Nitsche method} \cite{nitsche} yields
\begin{equation}
  \int_{\Omega_h} \nabla u \cdot \nabla v - \int_{\Gamma_h} \frac{\partial
  u}{\partial n} v + \int_{\Gamma} u \frac{\partial v}{\partial n} +
  \frac{\gamma}{h}  \int_{\Gamma} uv = \int_{\Omega_h} fv + \int_{\Gamma} g
  \frac{\partial v}{\partial n} + \frac{\gamma}{h}  \int_{\Gamma} gv,
	\quad\forall v \in H^1 (\Omega_h) 
  \label{Dir:var}
\end{equation}
with $\gamma > 0$. Here, $n$ on $\Gamma$ (resp. $\Gamma_h$) denotes the unit normal looking outwards from $\Omega$ (resp.
$\Omega_h$). Our finite element method is then based on the weak
formulation (\ref{Dir:var}) adding to it a ghost penalty stabilization to
assure the coerciveness of the bilinear form. The method will be fully presented in Subsection \ref{IntroDir} and analyzed in Subsections \ref{CoerDir} and \ref{ConvDir}. \red{Unlike the preliminary version in \cite{lozinski16}, the optimal error estimates are here proved under the natural assumptions on the regularity of the data: $f\in L^2(\Omega_h)$, $g\in H^{3/2}(\Gamma)$  (vs. $f\in H^1(\Omega_h)$ in \cite{lozinski16}).} We stress again that formulation (\ref{Dir:var}) does not contain integrals on $\Omega$. Otherwise, the resulting method is very close to the antisymmetric Nitsche CutFEM method from {\cite{burmanghost}}.

We turn now to the case of Neumann boundary conditions (\ref{Neu:P}). We start again by
extending $f$ from $\Omega$ to $\Omega_h$ and imagine (for the moment) that
(\ref{Neu:P}) can be solved on the extended domain $\Omega_h$ while still
imposing the boundary conditions on $\Gamma$:
\begin{equation}
  - \Delta u = f \text{ in } \Omega_h, \qquad \frac{\partial u}{\partial n} = g
  \text{ on } \Gamma . \label{Neu:Pext}
\end{equation}
Integration by parts over $\Omega_h$ and
imposing the boundary conditions weakly on $\Gamma$ would yield
\begin{eqnarray*}
  \int_{\Omega_h} \nabla u \cdot \nabla v - \int_{\Gamma_h} \frac{\partial
  u}{\partial n} v + \int_{\Gamma} \frac{\partial u}{\partial n} v
	=  \int_{\Omega_h} fv + \int_{\Gamma} gv
\end{eqnarray*}
Such a formulation does not seem to lead to a reasonable finite element
method. It is difficult to imagine a stabilization that would make the bilinear form on the left-hand side above coercive; \red{unlike the Dirichlet case, one cannot rely on the smallness of $u$ near $\Gamma$ to control the non-positive terms in the bilinear form. Note that a Nitsche method for the Neumann boundary conditions is proposed in \cite{juntunen09} in the mesh conforming case. However, the terms added there to the natural variational formulation tend to weaken its coercivenes, rather than to enhance it.}

Fortunately, we are able to devise an optimally convergent method by introducing a reconstruction $y = - \nabla u$ of the gradient of $u$ on the band of cut mesh elements $\Omega_h^{\Gamma}$, cf. Fig.~\ref{Dir:figdom} on the right, and by replacing the weak formulation above with the following
one
\begin{align}
  \int_{\Omega_h} \nabla u \cdot \nabla v  &+ \int_{\Gamma_h} y \cdot nv -
  \int_{\Gamma} y \cdot nv  + \gamma_{1} \int_{\Omega_h^{\Gamma}} (y + \nabla u) \cdot (z + \nabla v)
		&& \label{Neu:var} \\
   &= \int_{\Omega_h} fv + \int_{\Gamma} gv,
	&&\forall v \in H^1 (\Omega_h), z \in L^2 (\Omega_h^{\Gamma})^d
	\notag
\end{align}
with $\gamma_{1} > 0$. The terms multiplied by $\gamma_{1}$ serve to impose $y = - \nabla u$ on
$\Omega_h^{\Gamma}$. Our finite element method (\ref{Neu:Ph}), based on the weak formulation (\ref{Neu:var}) with the addition of grad-div and ghost stabilizations, will be fully presented in Subsection \ref{IntroRob} and analyzed in Subsections \ref{CoerRob} and \ref{ConvRob}, under the assumption of some extra regularity on the data: $f\in H^1(\Omega_h)$, $g\in H^{3/2}(\Gamma)$ so that $u\in H^3(\Omega)$. It involves the additional vector-valued unknown $y_h$, which results in some extra cost in practice, as compared with a simpler method for Dirichlet boundary conditions. Fortunately, this extra cost is negligible as $h\to 0$ since $y_h$ lives only on the mesh elements cut by $\Gamma$ which  constitute smaller and smaller proportion of the total number of elements as the meshes are refined. 

Our finite element methods for both Dirichlet and Neumann cases are presented in more detail in the
next section. Their well-posedness and the optimal error estimates are
proved in Section 3. We restrict ourselves here to $P_1$ continuous finite elements on
a triangular/tetrahedral mesh. An extension to higher-order elements $P_k$ \red{and a possible adaptation to Robin boundary conditions} are sketched in Section 4. We present some numerical experiments in Section 5, restricting ourselves to the $P_1$ elements. Some final conclusions are provided in section 6.

\section{Presentation of the methods}

\subsection{Dirichlet boundary conditions}\label{IntroDir}

We present first our discretization of Problem (\ref{Dir:P}). Recall that $\Th$ is a quasi-uniform mesh obtained from a larger background mesh retaining only the elements lying inside $\Omega$ or cut by $\Gamma$, cf. (\ref{ThOmegah}) and Fig.~\ref{Dir:figdom}, and $\Omega_h$ is the corresoinding domain with boundary $\Gamma_h$.
We inspire ourselves from the variational formulation (\ref{Dir:var}), introduce the finite element space
\begin{equation}\label{Dir:Vh}
  V_h = \{v_h \in H^1 (\Omega_h) : v_h |_T \in \mathbb{P}_1 (T)
  \  \forall T \in \mathcal{T}_h \}
\end{equation}
with $\mathbb{P}_1$ denoting the set of polynomials of degree $\le 1$, and introduce the following discrete problem:\\
Find $u_h \in V_h$ such that
\begin{equation}
  a_h (u_h, v_h) = \int_{\Omega_h} fv_h + \int_{\Gamma} g \frac{\partial
  v_h}{\partial n} + \frac{\gamma}{h}  \int_{\Gamma} gv_h \quad \forall v_h
  \in V_h \label{Dir:Ph}
\end{equation}
where
\begin{equation}\label{Dir:ah}
  a_h (u, v)  = \int_{\Omega_h} \nabla u \cdot \nabla v - \int_{\Gamma_h}
  \frac{\partial u}{\partial n} v + \int_{\Gamma} u \frac{\partial v}{\partial
  n} + \frac{\gamma}{h}  \int_{\Gamma} uv + \sigma h \sum_{E \in
  \mathcal{F}_{\Gamma}} \int_E \left[ \frac{\partial u}{\partial n} \right]
  \left[ \frac{\partial v}{\partial n} \right]
\end{equation}
and $\gamma, \sigma$ are some positive numbers properly chosen in a manner independent of $h$. 
The last term in (\ref{Dir:ah}) is the ghost penalty \cite{burmanghost}. It is crucial to assure the coerciveness of $a_h$.
The notations here are as follows: $[\cdot]$ stands for the jump over an internal facet of mesh $\Th$ and
\begin{eqnarray*}
  \mathcal{F}_{\Gamma} = \{E \text{ (an internal facet of } \mathcal{T}_h)
  \text{ such that } \exists T \in \mathcal{T}_h : T \cap \Gamma \neq
  \varnothing \text{ and } E \in \partial T\}
\end{eqnarray*}
The ghost penalty is thus a properly scaled sum of the jumps of the normal derivatives over all the internal facets of the mesh either cut by $\Gamma$ themselves or owned by a mesh element cut by $\Gamma$. 

\subsection{Neumann boundary conditions}\label{IntroRob}

We turn now to Problem (\ref{Neu:P}).
We inspire ourselves with the variational formulation (\ref{Neu:var}) and introduce a subspace of  $V_h$ (continuous $P_1$ finite elements on mesh $\Th$ as defined above) appropriate to the treatment of Neumann boundary conditions
\begin{equation}\label{Vhtild}
\red{\tilde{V}_h=\left\{v_h\in V_h : \int_{\Omega_h}=0\right\}}
\end{equation}
and an auxiliary finite element space
\[
  Z_h = \{z_h \in H^1 (\Omega^{\Gamma}_h)^d : z_h |_T \in \mathbb{P}_1 (T)^d
  \  \forall T \in \mathcal{T}_h^{\Gamma} \}
\]
Here $\mathcal{T}_h^{\Gamma}$ represents the cut elements of the mesh and $\Omega_h^{\Gamma}$ the corresponding subdomain of $\Omega_h$:
\begin{equation}
\mathcal{T}_h^{\Gamma} = \{T \in \mathcal{T}_h : T  \cap \Gamma = \varnothing\}
\text{ and }
\Omega_h^{\Gamma} = (\cup_{T \in\mathcal{T}_h^{\Gamma}} T)^{\circ}
\label{ThGamma}
\end{equation}
cf. Fig.~\ref{Dir:figdom}, right. We shall also denote by $\Gamma_h^i$ the internal boundary of
$\Omega_h^{\Gamma}$, i.e. the ensemble of the facets  separating $\Omega_h^{\Gamma}$ from the
mesh elements inside $\Omega$, so that $\partial\Omega_h^{\Gamma} = \Gamma_h \cup \Gamma_h^i$.

Our finite element problem is: Find $u_h \in \tilde{V}_h$, $y_h \in Z_h$ such that
\begin{equation}
  \label{Neu:Ph} a^N_h (u_h, y_h ; v_h, z_h) = \int_{\Omega_h} fv_h +
  \int_{\Gamma} gv_h + \gamma_{div} \int_{\Omega_h^{\Gamma}} f \Div z_h \quad
  \forall (v_h, z_h) \in \tilde{V}_h \times Z_h
\end{equation}
where
\begin{align}\label{Rob:ah}
  a^N_h (u, y ; v, z) = \int_{\Omega_h} \nabla u \cdot \nabla v &+
  \int_{\Gamma_h} y \cdot nv - \int_{\Gamma} y \cdot nv + \gamma_{div}\int_{\Omega_h^{\Gamma}} \Div y \Div z \\
   & + \gamma_{1} \int_{\Omega_h^{\Gamma}} (y + \nabla u) \cdot (z + \nabla
  v) + \sigma h \int_{\Gamma_h^i} \left[\ddn{u}\right] \left[\ddn{v}\right]
	\notag
\end{align}
with some positive numbers $\gamma_{div}$, $\gamma_{1}$ and $\sigma$ properly chosen in a manner
independent of $h$. In addition to the variational formulation (\ref{Neu:var}), we have introduced here
\begin{itemize}
	\item a grad-div stabilization (the terms multiplied by $\gamma_{div}$) in the vein of \cite{franca88}, which is consistent with the governing equations since $y=-\nabla u$ and thus $\Div y = -\Delta u = f$ on $\Omega_h^{\Gamma}$;
	\item a reduced ghost stabilization (the term multiplied by $\sigma h$ in (\ref{Rob:ah}). Contrary to the Dirichlet case (\ref{Dir:ah}), it is sufficient to penalize the normal derivative jumps only on the facets separating the cut elements from the uncut ones.
\end{itemize}
\begin{remark}\label{RemNeu} 
The introduction of the subspace $\tilde{V}_h$ is helpful to ensure the well-posedness of problem \ref{Rob:ah} (otherwise, if posed on $V_h\times Z_h$, the functions $u=1$, $y=0$ would be in the kernel of $a^N_h$ reflecting the fact that the exact solution is not unique without an additional constraint). We have chosen to impose this constraint on $\Omega_h$ rather than on $\Omega$ in line with our desire to avoid the integrals on $\Omega$, difficult to implement in practice.
\end{remark}

\section{The theoretical error analysis.}\label{theory}
\subsection{Geometrical assumptions and technical lemmas.}

The theoretical analysis of the methods presented above will be done under the following assumptions on the boundary $\Gamma$ and on the subdomain $\Omega_h^\Gamma$ covered by the cut mesh elements, as defined in (\ref{ThGamma}). Both these assumptions are typically satisfied if the mesh is sufficiently refined with respect to $\Gamma$.

\begin{assumption}\label{asm1}
The subdomain $\Omega_h^\Gamma$ can be covered by open sets $\mathcal{O}_k$, $k=1,\ldots,K$ and one can introduce on every $\mathcal{O}_k$ local coordinates $\xi_1,\xi_2$ in the 2D case, $d=2$ (resp. $\xi_1,\xi_2,\xi_3$ if $d=3$) such that
\begin{itemize}
	\item $\Gamma\cap\mathcal{O}_k$ is given by $\xi_d=0$ and $\Omega\cap\mathcal{O}_k$ by $\xi_d<0$;
	\item $\Omega_h^\Gamma\cap\mathcal{O}_k$ is given by $-a(\xi_1,\xi_2)<\xi_d<b(\xi_1,\xi_2)$ with some continuous non-negative functions $a$ and $b$;
	\item $b(\xi_1,\xi_2)+a(\xi_1,\xi_2) \le C_1h$ with some $C_1>0$;
	\item all the partial derivatives $\partial\xi_i/\partial x_j$ and $\partial x_i/\partial \xi_j$ are bounded by some $C_2>0$.
\end{itemize}
Moreover, each point of $\Omega_h^\Gamma$ is covered by at most $N_{int}$ sets $\mathcal{O}_k$. 	
\end{assumption}

\begin{assumption}\label{asm2}
  The boundary $\Gamma$ can be covered by element patches $\{\Pi_k \}_{k = 1,
  \ldots, N_{\Pi}}$ having the following properties:
  \begin{itemize}
    \item Each patch $\Pi_k$ is a connected set;
    \item Each $\Pi_k$ is composed of a mesh element $T_k$ lying inside $\Omega$ and some mesh elements cut by $\Gamma$, more precisely $\Pi_k = T_k \cup \Pi_k^{\Gamma}$ where $T_k\in\mathcal{T}_h$, $T_k\subset\bar\Omega$, $\Pi_k^{\Gamma}\subset\mathcal{T}_h^{\Gamma}$, and $\Pi_k^{\Gamma}$ contains at most $M$ mesh elements;    
    \item $\mathcal{T}_h^{\Gamma} = \cup_{k = 1}^{N_{\Pi}} \Pi_k^{\Gamma}$;
    \item $\Pi_k$ and $\Pi_l$ are disjoint if $k \neq l$.
  \end{itemize}
\end{assumption}

In what follows, we suppose both assumptions above to hold true and use the notation $C$ for positive constants (which can change from one instance to another) that depend only on $C_1,C_2,N_{int},M$ from the assumptions above and on the mesh regularity. We also recall that mesh $\Th$ is supposed quasi-uniform. 

\begin{figure}[htp]
\centerline{
\includegraphics[width=0.5\textwidth]{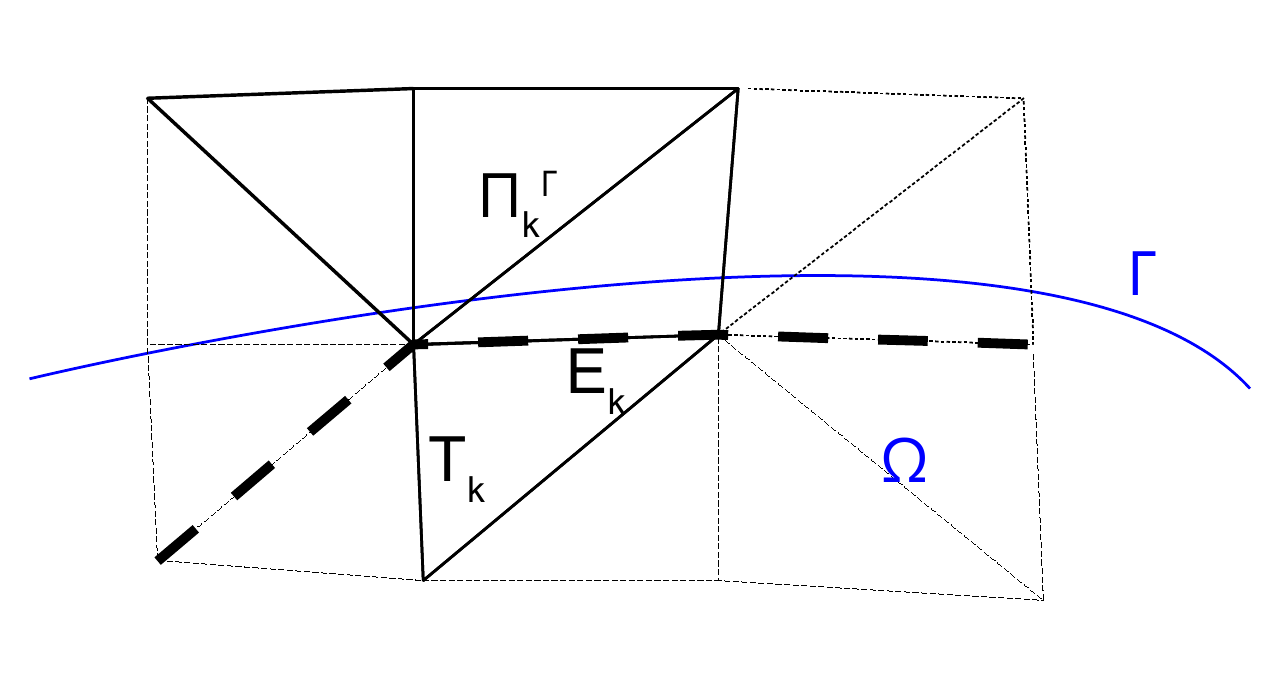}
}
\caption{Illustration of the construction of an element patch $\Pi_k$: a portion of a triangular mesh $\Th$ with the facets composing $\Pi_k$ represented by solid lines, the facets separating the cut triangles from the interior ones ($\Gamma_h^i$) represented by thick dashed lines, and the remaining facets in thin dashed lines. The patch $\Pi_k$ is composed of the interior element $T_k\subset\bar\Omega$ attached to a facet $E_k\subset\Gamma_h^i$ and several triangles cut by $\Gamma$ (the remaining part of $\Pi_k$ denoted by $\Pi_k^{\Gamma}$).}
\label{figpatch}
\end{figure}

Assumption \ref{asm1} is quite standard and essentially tells us that the boundary $\Gamma$ is smooth and not too wiggly on the length scale $h$, so that one can be sure that the band of cut mesh elements is of width $\sim h$.  Assumption \ref{asm2} is slightly more technical. Let us explain the construction of element patches evoked there, cf. Fig.~\ref{figpatch}.  Each patch $\Pi_k$ is assigned to a facet $E_k\subset\Gamma_h^i$ separating the cut elements from the interior ones. To form the patch $\Pi_k$, one takes a mesh element $T_k$ lying inside $\Omega$ and attached to $E_k$.  One then picks up several cut elements touching $E_k$ to form $\Pi_k^{\Gamma}$ and set $\Pi_k=T_k\cup \Pi_k^{\Gamma}$. If, again, the boundary $\Gamma$ is not too wiggly with respect to the mesh $\Th$, one can partition the cut elements between the patches $\Pi_k$ so that each patch contains a small  number of elements (typically from 2 to 4 in 2D, and slightly more in 3D).

We begin with two technical lemmas. The first one is in the vein of Poincar\'e inequalities taking into account the assumption that $\Omega_h^{\Gamma}$ is a strip of width $\sim h$ around  $\Gamma$.

\begin{lemma}\label{Prelim1}
  For all $v \in H^1 (\Omega_h^{\Gamma})$,
  \begin{equation}
    \|v\|_{0, \Omega_h^{\Gamma}} \le C \left( \sqrt{h} \|v\|_{0, \Gamma} +
    h|v|_{1, \Omega_h^{\Gamma}} \right) \label{Dir:eq:1}
  \end{equation}
  
\end{lemma}

\begin{proof} Consider the 2D case ($d=2$). Recalling Assumption \ref{asm1}, we can pass to the local coordinates $\xi_1,\xi_2$ on every set  $\mathcal{O}_k$  assuming that $\xi_1$ varies between $0$ and $L$, and to use the bounds on $\xi_2$ and on the mapping $(x_1,x_2)\mapsto(\xi_1,\xi_2)$ to write 
  \begin{align*}
    \|v\|_{0, \Omega^{\Gamma}_h(\xi_1) \cap \mathcal{O}_k}^2 &\le C\int_0^L \int_{-a(\xi_1)}^{b(\xi_1)}
    v^2 (\xi_1, \xi_2) d \xi_2 d \xi_1 = C\int_0^L \int_{-a(\xi_1)}^{b(\xi_1)} \left( v
    (\xi_1, 0) + \int_0^{\xi_2} \frac{\partial v}{\partial \xi_2} (\xi_1, t)
    {dt} \right)^2 d \xi_2 d \xi_1 \\ &\leqslant C\int_0^L \int_{-a(\xi_1)}^{b(\xi_1)} \left(
    v^2 (\xi_1, 0) +  \left( \int_0^{\xi_2} \frac{\partial
    v}{\partial \xi_2} (\xi_1, t) {dt} \right)^2 \right) d \xi_2 d \xi_1 \\
    & \leqslant C\int_0^L \left( hv^2 (\xi_1, 0) + h^2 \int_{-a(\xi_1)}^{b(\xi_1)}
    \left| \frac{\partial v}{\partial \xi_2} (\xi_1, t) \right|^2 {dt}
    \right)^{} d \xi_1 \\
		&\le Ch\|v\|^2_{0, \Gamma \cap \mathcal{O}_k} + 
    Ch^2 \left\| \nabla v \right\|^2_{0,\Omega^{\Gamma}_h \cap \mathcal{O}_k}
  \end{align*}
  Summing over all neighborhoods $\mathcal{O}_k$ gives (\ref{Dir:eq:1}). The proof in the 3D case is the same up to the change of notations.
\end{proof}

\begin{corollary}\label{Prelim12}
  For all $v \in H^2 (\Omega_h^{\Gamma})$,
  \begin{equation}
    \|v\|_{0, \Omega_h^\Gamma} \le C \left(
		h^{1/2}\|v\|_{0, \Gamma} + h^{3/2} \| \nabla v \|_{0, \Gamma} 
		+ h^{2} \| v \|_{2, \Omega^{\Gamma}_h }
	\right) \label{Dir:eq:12}
  \end{equation}
\end{corollary}

\begin{proof} Applying Lemma \ref{Prelim1} to $\nabla v$ we get
  \begin{equation}
    |v|_{1, \Omega_h^{\Gamma}} \le C \left( \sqrt{h} \|\nabla v\|_{0, \Gamma} +
    h|v|_{2, \Omega_h^{\Gamma}} \right) 
  \end{equation}
Substituting this into (\ref{Dir:eq:1}) leads to (\ref{Dir:eq:12}). 	
\end{proof}

\begin{lemma}\label{Prelim17}
  For all $v \in H^2 (\Omega_h^{\Gamma})$,
  \begin{equation}
    \|v\|_{0, \Gamma_h} \le C \left(
		\|v\|_{0, \Gamma} + h \| \nabla v \|_{0, \Gamma} 
		+ h^{3/2} \| v \|_{2, \Omega^{\Gamma}_h }
	\right) \label{Dir:eq:17}
  \end{equation}
\end{lemma}

\begin{proof} As in the proof of Lemma \ref{Prelim1}, we only consider the 2D case. Using again Assumption \ref{asm1}, we can represent $v|_{\Gamma_h}$ inside every $\mathcal{O}_k$  by a Taylor expansion with the integral remainder
\begin{align*} \|v\|_{0,^{} \Gamma_h \cap \mathcal{O}_k}^2 &\leqslant C \int_0^L v^2
   (\xi_1, b (\xi_1)) d \xi_1 = C \int_0^L \left( v (\xi_1, 0) + b (\xi_1)
   \frac{\partial v}{\partial \xi_2} (\xi_1, 0) + \int_0^{b (\xi_1)}
   \frac{\partial^2 v}{\partial \xi^2_2} (\xi_1, t) (b (\xi_1) - t) {dt}
   \right)^2 d \xi_1 \\
	&\leqslant C \int_0^L  \left( v^2 (\xi_1, 0) +   h^2 \frac{\partial v}{\partial \xi_2} (\xi_1, 0)^2 
	+ h^3 \left|
   \int_0^{b(\xi_1)} \frac{\partial^2 v}{\partial \xi^2_2} (\xi_1, t)^2 {dt}
   \right| \right) d \xi_1 \\
	&\leqslant C (\|v\|^2_{0, \Gamma \cap \mathcal{O}_k}
   + h^2 \| \nabla v \|^2_{0, \Gamma \cap \mathcal{O}_k} + h^3 \| v \|^2_{2,
   \Omega^{\Gamma}_h \cap \mathcal{O}_k}) 
	\end{align*}
Summing this over all $\mathcal{O}_k$ gives (\ref{Dir:eq:17}).
\end{proof}
	
\begin{lemma}\label{Prelim2}
  For all $v \in H^1(\Omega_h^\Gamma)$,
  \begin{equation}
    \sum_{E \in \mathcal{F}_{\Gamma}} \|v \|_{0, E}^2 \leq C (\|v \|_{0,
    \Gamma}^2 + h|v |_{1, \Omega_h^{\Gamma}}^2) . \label{Dir:prop2}
  \end{equation}
  and
  \begin{equation}
    \|v \|_{0, \Gamma_h}^2 \leq C (\|v \|_{0, \Gamma}^2 + h|v |_{1,
    \Omega_h^{\Gamma}}^2) . \label{Dir:prop3}
  \end{equation}
\end{lemma}

\begin{proof}
  Let $E$ be a mesh facet belonging to a mesh element $T\in\Th$. Recall the well-known trace inequality
  \begin{equation*}
    \|v\|_{0, E}^2 \le C \left( \frac{1}{h} \|v\|_{0, T}^2 + h|v|_{1, T}^2
    \right) 
  \end{equation*}
  Summing   this over all $E \in \mathcal{F}_{\Gamma}$ gives
	$$
	  \sum_{E \in \mathcal{F}_{\Gamma}} \|v_h \|_{0, E}^2 \leq C 
		\left( \frac{1}{h} \|v\|_{0, \Omega_h^\Gamma}^2 + h|v|_{1, \Omega_h^\Gamma}^2 \right) 
	$$
	leading, in combination with  (\ref{Dir:eq:1}), to (\ref{Dir:prop2}). The proof of (\ref{Dir:prop3}) is
  similar: it suffices to take the sum over the facets composing $\Gamma_h$.
\end{proof}

\subsection{Dirichlet boundary conditions: coerciveness of $a_h$. }\label{CoerDir}

We turn to the study of method (\ref{Dir:Ph})--(\ref{Dir:ah}) for the problem with Dirichlet boundary conditions. Our first goal is to prove the coerciveness of the bilinear form $a_h$ uniformly in $h$. The proof of this result, cf. Lemma \ref{LemDir:coer}, will be based on the following

\begin{lemma}
  \label{LemDir:prop1}For any $\beta > 0$ one can choose $0 < \alpha < 1$ depending only on the mesh
  regularity and on parameter $M$ from Assumption \ref{asm2} such that 
  \begin{equation}
    |v_h |_{1, \Omega_h^{\Gamma}}^2 \leqslant \alpha |v_h |_{1, \Omega_h}^2 +
    \beta h \sum_{E \in \mathcal{F}_{\Gamma}} \left\| \left[ \frac{\partial
    v_h}{\partial n} \right] \right\|_{0, E}^2 \label{Dir:prop1}
  \end{equation}
	for all $v_h \in V_h$.
\end{lemma}

\begin{proof}
  Choose any $\beta > 0$, consider the decomposition of $\Omega_h^\Gamma$ in element patches $\{\Pi_k \}$ as in Assumption \ref{asm2}, cf. also Fig.~\ref{figpatch}. Introduce
  \[
    \label{Dir:alph} \alpha := \max_{\Pi_k, v_h \neq 0}  \frac{|v_h |_{1,
    \Pi_k^{\Gamma}}^2 - \beta h \sum_{E \in \mathcal{F}_k} \left\| \left[
    \frac{\partial v_h}{\partial n} \right] \right\|_{0, E}^2}{|v_h |_{1,
    \Pi_k}^2}
  \]
  where the maximum is taken over all the possible configurations of a patch
  $\Pi_k$ allowed by the mesh regularity and over all the piecewise linear
  functions on $\Pi_k$. The subset $\mathcal{F}_k \subset
  \mathcal{F}_{\Gamma}$ gathers the facets internal to $\Pi_k$. Note that the
  quantity under the $\max$ sign in (\ref{Dir:alph}) is invariant under the
  scaling transformation $x \mapsto hx$ and is homogeneous with respect to
  $v_h$. Recall also that the patch $\Pi_k$ contains at most $M$ elements. Thus, the maximum is indeed attained since it is taken over a bounded set in a finite dimensional space (one can restrict the maximization to patches of diameter of order 1, since the quantity to maximize is invariant under rescaling). 
	
	Clearly, $\alpha \le 1$. Supposing
  $\alpha = 1$ would lead to a contradiction. Indeed, if $\alpha = 1$ then we
  can take $\Pi_k$, $v_h$ yielding this maximum and suppose without loss of
  generality $|v_h |_{1, \Pi_k} = 1$. We observe then
  \begin{eqnarray*}
    |v_h |_{1, T_k}^2 + \beta h \sum_{E \in \mathcal{F}_k} \left\| \left[
    \frac{\partial v_h}{\partial n} \right] \right\|_{0, E}^2 = 0
  \end{eqnarray*}
	since $|v_h |_{1,\Pi_k}^2=|v_h |_{1,T_k}^2+|v_h |_{1,\Pi_k^\Gamma}^2$.
  This implies $\nabla v_h = 0$ on $T_k$ and $[\ddn{v_h}] = 0$ on all $E \in
  \mathcal{F}_k$, thus $\nabla v_h = 0$ on $\Pi_k$, which contradicts $|v_h
  |_{1, \Pi_k} = 1$. 
	
	This proves $\alpha < 1$. We have thus
  \begin{eqnarray*}
    |v_h |_{1, \Pi_k^{\Gamma}}^2 \le \alpha |v_h |_{1, \Pi_k}^2 + \beta h
    \sum_{E \in \mathcal{F}_k} \left\| \left[ \frac{\partial v_h}{\partial n}
    \right] \right\|_{0, E}^2
  \end{eqnarray*}
  for all $v_h\in V_h$ and all the admissible patches $\Pi_k$. Summing this over $\Pi_k$,
  $k = 1, \ldots, N_{\Pi}$ yields (\ref{Dir:prop1}).
\end{proof}

\begin{lemma}\label{LemDir:coer}
  Provided $\sigma$ is sufficiently big, there exists an $h$-independent
  constant $c > 0$ such that 
  \begin{eqnarray*}
    a (v_h, v_h) \ge c \triple{v_h}_h^2 \qquad \text{with} \qquad
    \triple{v}_h^2 = |v|_{1, \Omega_h}^2 + \frac{1}{h} \|v\|_{0, \Gamma}^2
  \end{eqnarray*}
	for all $v_h \in V_h$.
\end{lemma}

\begin{proof}
  Let $v_h \in V_h$ and $B_h$ be the strip between $\Gamma$ and $\Gamma_h$, i.e. $B_h = \Omega_h \setminus \Omega$ and $\partial B_h = \Gamma \cup \Gamma_h$. Recall that we assume the normal $n$ to look outward from $\Omega_h$ (resp. $\Omega$) on  $\Gamma_h$ (resp. $\Gamma$). The outward looking normal on $\partial B_h$ coincides with $n$ on  $\Gamma_h$ and with $-n$ on  $\Gamma$. We have thus  
	\begin{align} \int_{\Gamma_h} \frac{\partial v_h}{\partial n} v_h - \int_{\Gamma} 
   \frac{\partial v_h}{\partial n} v_h &= \int_{\partial B_h}  \frac{\partial
   v_h}{\partial n} v_h = \sum_{T \in \mathcal{T}_h^{\Gamma}} \int_{\partial
   B_h \cap T}  \frac{\partial v_h}{\partial n} v_h 
	\notag \\
	 &\quad (\text{using }\partial T=(\partial B_h \cap T) \cup (B_h \cap \partial T) 
	   \text{ and } \Delta v_h=0 \text{ on } T)
		\label{LemDir:eq1}\\
	 &= \sum_{T \in
   \mathcal{T}_h^{\Gamma}} \left[ \int_T | \nabla v_h |^2 - \int_{B_h \cap
   \partial T}  \frac{\partial v_h}{\partial n} v_h \right] = \int_{B_h} |
   \nabla v_h |^2 - \sum_{F \in \mathcal{F}_{\Gamma}} \int_{F \cap B_h} v_h
   \left[ \frac{\partial v_h}{\partial n} \right] 
	  \notag
	\end{align} 
	Substituting this into the definition (\ref{Dir:ah}) of $a_h$ yields
  \begin{eqnarray*}
    a_h (v_h, v_h) = \int_{\Omega_h} | \nabla v_h |^2 - \int_{B_h} | \nabla
    v_h |^2 + \sum_{F \in \mathcal{F}_{\Gamma}} \int_{F \cap B_h} v_h \left[
    \frac{\partial v_h}{\partial n} \right] + \frac{\gamma}{h}  \int_{\Gamma}
    v_h^2 + \sigma h \sum_{E \in \mathcal{F}_{\Gamma}} \int_E \left[
    \frac{\partial v_h}{\partial n} \right]^2
  \end{eqnarray*}
  Noting that
  $B_h \subset \Omega_h^{\Gamma}$ we can use (\ref{Dir:prop1}) combined with
  the Young inequality (for any $\varepsilon > 0$) and (\ref{Dir:prop2}) to
  write  
  \begin{align*}
    a (v_h, v_h) & \geqslant (1 - \alpha) |v_h |_{1, \Omega_h}^2 + \left(
    \sigma - \beta - \frac{1}{2 \varepsilon} \right) h \sum_{E \in
    \mathcal{F}_{\Gamma}} \left\| \left[ \frac{\partial v_h}{\partial n}
    \right] \right\|_{0, E}^2 - \frac{\varepsilon}{2 h}  \sum_{E \in
    \mathcal{F}_{\Gamma}} \|v_h \|_{0, E}^2 + \frac{\gamma}{h} \|v_h \|_{0,
    \Gamma}^2\\
    & \geqslant \left( 1 - \alpha - \frac{\varepsilon C}{2} \right) |v_h
    |_{1, \Omega_h}^2 + \left( \sigma - \beta - \frac{1}{2 \varepsilon}
    \right) h \sum_{E \in \mathcal{F}_{\Gamma}} \left\| \left[ \frac{\partial
    v_h}{\partial n} \right] \right\|_{0, E}^2 + \frac{\gamma - \varepsilon C /
    2}{h} \|v_h \|_{0, \Gamma}^2
  \end{align*}
  Taking $\varepsilon$ sufficiently small and $\sigma$ sufficiently big this
  bounds $a (v_h, v_h)$ from below by $c \triple{v_h}_h^2$ as claimed.
\end{proof}

\subsection{Dirichlet boundary conditions: the error estimates. }\label{ConvDir}

We can now establish an optimal $H^1$ error estimate for the method (\ref{Dir:Ph})--(\ref{Dir:ah}) using the coerciveness of $a_h$ provided by Lemma \ref{LemDir:coer}. An $L^2$ error estimate will follow in Theorem \ref{ThDir:L2}.
\begin{theorem}
  \label{ThDir:aprio}Suppose \red{$f \in L^2 (\Omega_h)$, $g \in H^{3 / 2}  (\Gamma)$} and let $u \in H^2 (\Omega)$ be the solution to (\ref{Dir:P}),
  $u_h \in V_h$ be the solution to (\ref{Dir:Ph})--(\ref{Dir:ah}). Provided $\sigma$ is
  sufficiently big, there exists an $h$-independent constant $C > 0$ such that
  \begin{equation}
    |u - u_h |_{1, \Omega} + \frac{1}{\sqrt{h}}  \|u - u_h \|_{0, \Gamma} \le
    Ch (\|f\|_{\red{0, \Omega_h}} +\|g\|_{\red{3 / 2, \Gamma}}) . \label{Dir:aprio}
  \end{equation}
\end{theorem}

\begin{proof}
  Under the Theorem's assumptions, the solution to (\ref{Dir:P}) is indeed in
  $H^2 (\Omega)$ and it can be extended to a function $\tilde{u} \in H^2
  (\Omega_h)$ such that $\tilde{u} = u$ on $\Omega$ and $\| \tilde{u} \|_{2,
  \Omega_h} \le C (\|f\|_{0, \Omega} +\|g\|_{3 / 2, \Gamma})$, cf. \cite{adams03}. Clearly,
  $\tilde{u}$ satisfies
  \begin{eqnarray*}
    a_h (\tilde{u}, v_h) = (\tilde{f}, v_h)_{L^2 (\Omega_h)} + \int_{\Gamma} g
    \frac{\partial v_h}{\partial n} + \frac{\gamma}{h}  \int_{\Gamma} gv_h
    \quad \forall v_h \in V_h
  \end{eqnarray*}
  with $\tilde{f} := - \Delta \tilde{u}$. It entails a Galerkin orthogonality relation 
	\begin{equation}
	 a_h  (\tilde{u} - u_h, v_h) = \int_{\Omega_h}
  (\tilde{f} - f) v_h,\quad \forall v_h\in V_h
	\label{Dir:GO}
	\end{equation} 
	We have then using the standard
  nodal interpolation $I_h : C (\bar{\Omega}_h) \to V_h$ and recalling Lemma \ref{LemDir:coer}
  \begin{align*}
    \frac{1}{c} \triple{u_h - I_h  \tilde{u}}_h & \le \sup_{v_h \in V_h} 
    \frac{a_h  (u_h - I_h  \tilde{u}, v_h)}{\triple{v_h}_h} = \sup_{v_h \in
    V_h}  \frac{a_h (e_u, v_h) + (f - \tilde{f}, v_h)_{L^2
    (\Omega_h)}}{\triple{v_h}_h}\\
    & 
  \end{align*}  
  with $e_u = \tilde{u} - I_h  \tilde{u}$. Using the definition (\ref{Dir:ah}) of $a_h$, 
  we can now bound term by term  
  \begin{align*}
    a_h (e_u, v_h) &\leqslant |e_u |_{1, \Omega_h} |v_h |_{1, \Omega_h} +
    \left\| \frac{\partial e_u}{\partial n} \right\|_{0, \Gamma_h} \|v_h \|_{0,
    \Gamma_h} + \|e_u \|_{0, \Gamma} \left\| \frac{\partial v_h}{\partial n}
    \right\|_{0, \Gamma} + \frac{\gamma}{h} \|e_u \|_{0, \Gamma} \|v_h \|_{0,
    \Gamma} \\
		&\qquad
    + \sigma h \sum_{E \in \mathcal{F}_{\Gamma}} \left\| \left[ \frac{\partial
    e_u}{\partial n} \right] \right\|_{0, E} \left\| \left[ \frac{\partial
    v_h}{\partial n} \right] \right\|_{0, E}\\
    &\leqslant C\left( |e_u |_{1, \Omega_h}^2 + h \left\| \frac{\partial
    e_u}{\partial n} \right\|^2_{0, \Gamma_h} + \frac{1}{h} \|e_u \|^2_{0,
    \Gamma} + h \sum_{E \in \mathcal{F}_{\Gamma}} \left\| \left[ \frac{\partial
    e_u}{\partial n} \right] \right\|_{0, E}^2 \right)^{\frac{1}{2}}
    \triple{v_h}_h
  \end{align*}  
  We have used here Lemma \ref{Prelim2} to bound $h \sum_{E \in \mathcal{F}_{\Gamma}} \left\| \left[ \frac{\partial
  v_h}{\partial n} \right] \right\|_{0, E}^2$ and $\|v_h \|_{0, \Gamma_h}$.
  This entails thanks to the usual interpolation estimates 
  \begin{eqnarray*}
    a_h (e_u, v_h) \leqslant Ch | \tilde{u} |_{2, \Omega_h} \triple{v_h}_h
  \end{eqnarray*}
  Moreover, $f = \tilde{f}$ on $\Omega$ so that
  \begin{eqnarray*}
    (f - \tilde{f}, v_h)_{L^2 (\Omega_h)} \leqslant \|f - \tilde{f} \|_{0,
    \Omega_h\setminus\Omega} \|v_h \|_{0, \Omega_h\setminus\Omega} \le 
		\red{Ch \|f - \tilde{f} \|_{0, \Omega_h} \triple{v_h}_h}
  \end{eqnarray*}
	\red{thanks to Lemma \ref{Prelim1} (recall that $\Omega_h\setminus\Omega\subset\Omega_h^\Gamma$).} 
  We conclude
  \begin{eqnarray*}
    \triple{u_h - I_h  \tilde{u}}_h 
		\le Ch (| \tilde{u} |_{2, \Omega_h} +\|f - \tilde{f} \|_{0, \Omega_h})
		\le Ch (| \tilde{u} |_{2, \Omega_h} +\|f\|_{0, \Omega_h} + \|\Delta\tilde{u}\|_{0, \Omega_h})
		\le Ch (| \tilde{u} |_{2, \Omega_h} +\|f\|_{0, \Omega_h} )
  \end{eqnarray*}
	Combining the estimates above with the triangle inequality proves
  $\triple{u_h - \tilde{u}}_h \le Ch (\|f\|_{0, \Omega_h} +\|g\|_{3 / 2,
  \Gamma})$, as claimed.
\end{proof}

\begin{remark}
The proof above does not rely on a solution to the non-standard boundary value problem (\ref{Dir:Pext}) on $\Omega_h$. We rather use the well defined solution $u$ to problem (\ref{Dir:P}) and extend it to $\Omega_h$. 
\end{remark}

The following theorem gives an $L^2$ estimate for the error. It is
sub-optimal, although the numerical experiments reveal the optimal convergence
rate $O (h^2)$, similar to the state of the art in the study of the
non-symmetric Nitsche method.

\begin{theorem}\label{ThDir:L2}
  Under the assumptions of Theorem \ref{ThDir:aprio}, there exists an
  $h$-independent constant $C > 0$ such that
  \begin{equation}
    \|u - u_h \|_{0, \Omega} \le Ch^{3 / 2}  \left(\|f\|_{\red{0, \Omega_h}} +\|g\|_{\red{3 / 2, \Gamma}}\right) . \label{Dir:aprioL2}
  \end{equation}
\end{theorem}

\begin{proof}
  Let us introduce $w : \Omega \to \mathbb{R}$ such that
  \begin{equation}
     - \Delta w = u - u_h  \text{ in } \Omega, \qquad w = 0
    \text{ on } \Gamma .
    \label{Dir:z}
  \end{equation}
  By elliptic regularity, $\|w\|_{2, \Omega} \leq C \|u - u_h \|_{0, \Omega}$.
  Let $\tilde{w}$ be an extension of $w$ from $\Omega$ to $\Omega_h$
  preserving the $H^2$ norm estimate. Applying Corollary \ref{Prelim12} to $\tilde{w}$ and Lemma \ref{Prelim1} to  $\nabla \tilde{w}$ yields
  $$\| \tilde{w} \|_{0, \Omega_h^{\Gamma}} \le Ch^{3/2} \|u - u_h \|_{0, \Omega}
	\quad\text{ and }\quad
  | \tilde{w} |_{1, \Omega_h^{\Gamma}} \le C \sqrt{h}  \|u - u_h \|_{0,
  \Omega}$$ 
	Similarly, applying Lemma \ref{Prelim17} to $\tilde{w}$ yields $$\| \tilde{w} \|_{0, \Gamma_h} \le Ch \|u - u_h \|_{0, \Omega}$$
  Taking $w_h = I_h  \tilde{w}$, we can summarize the bounds above together with the interpolation estimates as
  \begin{align}
    | \tilde{w} - w_h |_{1, \Omega_h}  + \sqrt{h} \left\| \frac{\partial
    (\tilde{w} - w_h)}{\partial n} \right\|_{0, \Gamma \cup \Gamma_h} &+
    \frac{1}{\sqrt{h}}  \| \tilde{w} - w_h \|_{0, \Gamma \cup \Gamma_h}
    \notag\\
     &+ \sqrt{h} | \tilde{w} |_{1, \Omega^{\Gamma}_h} + \| \tilde{w} \|_{0,
    \Gamma_h} + \frac{1}{\sqrt{h}}\|w_h \|_{0, \Omega^{\Gamma}_h} \le Ch \|u - u_h \|_{0,
    \Omega} \label{Dir:zest}
  \end{align}
  Using Galerkin orthogonality (\ref{Dir:GO}) and estimates (\ref{Dir:zest}) we arrive at (recall $B_h = \Omega_h \setminus \Omega$)
  \begin{eqnarray*}
    &  & \hspace{-15mm} \|u - u_h \|_{0, \Omega}^2 = \int_{\Omega} \nabla (u
    - u_h) \cdot \nabla w - \int_{\Gamma} (u - u_h)  \frac{\partial
    w}{\partial n}\\
    & = & a_h  (\tilde{u} - u_h, \tilde{w} - w_h) + \int_{\Gamma_h}
    \frac{\partial (\tilde{u} - u_h)}{\partial n}  \tilde{w} - 2 \int_{\Gamma}
    (u - u_h)  \frac{\partial w}{\partial n} - \int_{B_h} \nabla (\tilde{u} -
    u_h) \cdot \nabla \tilde{w} + \int_{B_h} (\tilde{f} - f) w_h\\
    & \le & C \left( \triple{\tilde{u} - u_h}_h + \left\| \frac{\partial
    (\tilde{u} - u_h)}{\partial n} \right\|_{0, \Gamma_h} + \frac{1}{h} \|u -
    u_h \|_{0, \Gamma} + \frac{1}{\sqrt{h}} | \tilde{u} - u_h |_{1,
    \Omega_h^{\Gamma}} + \sqrt{h}\| \tilde{f} - f\|_{0, \Omega_h^\Gamma} \right) h \|u - u_h \|_{0,
    \Omega}
  \end{eqnarray*}
  which gives the announced error estimate in $L^2 (\Omega)$ norm thanks to the error estimate in the triple norm $\triple{\cdot}$ already established in the proof of Theorem \ref{ThDir:aprio}. We also recall $\| \tilde{f}\|_{0, \Omega_h^\Gamma}=\| \Delta\tilde{u} \|_{0, \Omega_h}\le C\left(\|f\|_{0, \Omega_h} +\|g\|_{3/2, \Gamma}\right)$. 
\end{proof}

\subsection{Neumann boundary conditions: coerciveness of $a^N_h$.}\label{CoerRob}
We turn to the study of method (\ref{Neu:Ph})--(\ref{Rob:ah}) for the problem with Neumann boundary conditions. Our first goal is to prove the coerciveness of the bilinear form $a^N_h$ uniformly in $h$. 
To this end, we note that $a^N_h$ can be rewritten using the divergence Theorem as
\begin{align}
  \label{Neu:ah} a^N_h (u, y ; v, z) & = \int_{\Omega_h} \nabla u \cdot \nabla
  v + \int_{B_h} (v\Div y + y \cdot \nabla v) 
	+ \gamma_{div}  \int_{\Omega_h^{\Gamma}} \Div y \Div z\\
  & \quad + \gamma_{1} \int_{\Omega_h^{\Gamma}} (y + \nabla u) \cdot (z + \nabla v) +
  \sigma h \int_{\Gamma_h^i} \left[\ddn u\right] \left[\ddn v\right]
  \notag
\end{align}
provided $y$ and $v$ are of regularity $H^1$. We have denoted here again  $B_h = \Omega_h \setminus \Omega$. The following lemma will allow
us to control $\int_{B_h} y \cdot \nabla v$ while the other term $\int_{B_h} v\Div y$ can be controlled thanks to the grad-div stabilization.

\begin{lemma}
  \label{LemNeu:prop1} For any $\beta > 0$, there exist $0 < \alpha < 1$ and $\delta>0$ 
  depending only on the mesh regularity and on parameter $M$ in Assumption \ref{asm2} such that
  \begin{equation}
    \left| \int_{B_h} z_h \cdot \nabla v_h \right| \leqslant \alpha |v_h |_{1,
    \Omega_h}^2 + \delta \|z_h + \nabla v_h \|^2_{0, \Omega_h^{\Gamma}} + \beta
    h \left\|\left[\ddn{v_h}\right]\right\|_{0, \Gamma_h^i}^2 \label{Neu:prop1}
  \end{equation}
  for all $v_h \in V_h, z_h \in Z_h$.
\end{lemma}

\begin{proof}
  The boundary $\Gamma$ can be covered by element patches $\{\Pi_k \}_{k = 1,
  \ldots, N_{\Pi}}$ as in Assumption \ref{asm2}. Choose any $\beta > 0$ and consider
  \begin{equation}
    \label{Neu:alph} \alpha := \max_{\Pi_k, z_h, v_h \neq 0}  \frac{\|z_h
    \|_{0, \Pi_k^{\Gamma}} | v_h |_{1, \Pi_k^{\Gamma}} - \beta \|z_h + \nabla
    v_h \|^2_{0, \Pi_k^{\Gamma}} - \beta h \left\|\left[\ddn{v_h}\right]\right\|^2_{0, \partial
    T_k \cap \partial \Pi_k^{\Gamma}}}
		{\frac 12 \|z_h\|_{0, \Pi_k^{\Gamma}}^2 + \frac 12|v_h |_{1, \Pi_k}^2}
  \end{equation}
  where the maximum is taken over all the possible configurations of a patch
  $\Pi_k$ allowed by the mesh regularity and over all the piecewise linear
  functions $v_h$ and $z_h$ on $\Pi_k$. Note that the quantity under the
  $\max$ sign in (\ref{Neu:alph}) is invariant under the scaling
  transformation $x \mapsto hx$, $z_h \mapsto \frac{1}{h} z_h$, $v_h \mapsto v_h$ and is
  homogeneous with respect to $v_h$, $z_h$. Thus, the maximum is indeed attained
  since it is taken over a bounded set in a finite dimensional space.
	
	Clearly, $\alpha \leqslant 1$.
  Supposing $\alpha=1$ would lead to a contradiction. Indeed, if $\alpha=1$, we can
  then take $\Pi_k$, $v_h$, $z_h$ yielding this maximum (in particular, $|v_h |_{1, \Pi_k}^2+\|z_h\|_{0, \Pi_k^\Gamma}^2 >0$). We observe then
	$$   
	  \frac 12 |v_h |_{1, \Pi_k}^2 + \|z_h\|_{0, \Pi_k^{\Gamma}} | v_h |_{1, \Pi_k^{\Gamma}}
		+ \frac 12 \|z_h\|_{0, \Pi_k^{\Gamma}}^2
		+ \beta \|z_h + \nabla v_h \|^2_{0,
    \Pi_k^{\Gamma}} + \beta h \left\|\left[\ddn{v_h}\right]\right\|_{0, \partial T_k \cap
    \partial \Pi_k^{\Gamma}}^2 = 0
 $$
and consequently (recall $|v_h |_{1,\Pi_k}^2=|v_h |_{1,T_k}^2+|v_h |_{1,\Pi_k^\Gamma}^2$)
  \begin{equation}
    \label{Neu:alph1} \frac 12 |v_h |_{1, T_k}^2 + \beta \|z_h + \nabla v_h \|^2_{0,
    \Pi_k^{\Gamma}} + \beta h \left\|\left[\ddn{v_h}\right]\right\|_{0, \partial T_k \cap
    \partial \Pi_k^{\Gamma}}^2 = 0
  \end{equation}
  This implies $\|z_h + \nabla v_h \|_{0, \Pi_k^{\Gamma}}$=0, i.e. $\nabla v_h
  = - z_h$ on $\Pi_k^{\Gamma}$. Since $\nabla v_h$ is piecewise constant and
  $z_h$ is continuous, it means that $\nabla v_h = - z_h =$const on
  $\Pi_k^{\Gamma}$. We also have $\left\|\left[\ddn{v_h}\right]\right\|_{0, \partial T_k \cap
  \partial \Pi_k^{\Gamma}} = 0$, which implies $\nabla v_h =$const on the
  whole $\Pi_k$. Since, by (\ref{Neu:alph1}), $\nabla v_h =0$ on $T_k$, we have finally $\nabla v_h = 0$  on
  $\Pi_k$ and $z_h =0$ on  $\Pi_k^{\Gamma}$, which is in contradiction with $|v_h |_{1, \Pi_k}^2+\|z_h\|_{0, \Pi_k^\Gamma}^2 >0$.
	
Thus $\alpha < 1$ and
\begin{eqnarray*}
  \|z_h \|_{0, \Pi_k^{\Gamma}} |v_h |_{1, \Pi_k^{\Gamma}} \leqslant
  \frac{\alpha}{2} \|z_h \|_{0, \Pi_k^{\Gamma}}^2 + \frac{\alpha}{2} |v_h
  |_{1, \Pi_k}^2 + \beta \|z_h + \nabla v_h \|^2_{0, \Pi_k^{\Gamma}} + \beta h
  \left\|\left[\ddn{v_h}\right]\right\|^2_{0, \partial T_k \cap \partial \Pi_k^{\Gamma}}
\end{eqnarray*}
for all $v_h, z_h$ and all admissible patches $\Pi_k$. We now observe
\begin{align*}
  \left| \int_{B_h} z_h \cdot \nabla v_h \right| &\leqslant \sum_k \left|
  \int_{B_h \cap \Pi_k^{\Gamma}} z_h \cdot \nabla v_h \right| \leqslant \sum_k
  \|z_h \|_{0, \Pi_k^{\Gamma}} |v_h |_{1, \Pi_k^{\Gamma}}
\\
  & \leqslant \frac{\alpha}{2} \|z_h \|_{0, \Omega_h^{\Gamma}}^2 +
   \frac{\alpha}{2} |v_h |_{1, \Omega_h}^2 + \beta \|z_h + \nabla v_h \|^2_{0,
   \Omega_h^{\Gamma}} + \beta h \left\|\left[\ddn{v_h}\right]\right\|^2_{0, \Gamma_h^i} 
\\
  &  \leqslant \alpha \left( 1 + \frac{\varepsilon}{2} \right) |v_h |_{1,
   \Omega_h}^2 + \left( \beta + \frac{\alpha}{2} + \frac{\alpha}{2\varepsilon} \right)  \|z_h +
   \nabla v_h \|^2_{0, \Omega_h^{\Gamma}} + \beta h \left\|\left[\ddn{v_h}\right]\right\|^2_{0,
   \Gamma_h^i} 
\end{align*}
for any $\varepsilon > 0$.
We have used here the following estimate valid  by Young inequality
\[ \|z_h \|_{0, \Omega_h^{\Gamma}}^2 \leqslant \|z_h + \nabla v_h \|_{0,
   \Omega_h^{\Gamma}}^2 + \| \nabla v_h \|_{0, \Omega_h^{\Gamma}}^2 - 2 (z_h +
   \nabla v_h, \nabla v_h)_{0, \Omega_h^{\Gamma}} \]
\[ \leqslant \left( 1 + \frac{1}{\varepsilon} \right) \|z_h + \nabla v_h \|_{0, \Omega_h^{\Gamma}}^2
   +  (1 + \varepsilon)|v_h |_{1, \Omega_h}^2 \]  	
Taking $\varepsilon$ sufficiently small, redefining $\alpha$ as $\alpha \left( 1 + \frac{\varepsilon}{2} \right)$ and 
putting $\delta=\left( \beta + \frac{\alpha}{2} + \frac{\alpha}{2\varepsilon} \right)$ we obtain (\ref{Neu:prop1}).
\end{proof}

\begin{lemma}\label{LemRob:coer}
  Provided $\gamma_{div}, \gamma_{1}$ are sufficiently big, there exists an
  $h$-independent constant $c > 0$ such that 
  \begin{eqnarray*}
    a^N_h (v_h, z_h ; v_h, z_h) \ge c \triple{v_h, z_h}_h^2,
		\quad \forall v_h \in \tilde{V}_h, z_h \in  Z_h
  \end{eqnarray*}
  with
  \begin{eqnarray*}
    \qquad \triple{v, z}_h^2 = |v|_{1, \Omega_h}^2 + \| \Div z \|^2_{0, \Omega_h^{\Gamma}} 
		 + \|z + \nabla v \|^2_{0, \Omega_h^{\Gamma}} + h \left\|\left[\ddn
    v\right]\right\|_{0, \Gamma_h^i}
  \end{eqnarray*}
\end{lemma}

\begin{proof}
  Expression (\ref{Neu:ah}) for $a^N_h$ implies for all $v_h \in \tilde{V}_h$, $z_h \in Z_h$ 
  \begin{align*}
    a^N_h (v_h, z_h ; v_h, z_h) = |v_h |_{1, \Omega_h}^2 &
		+ \int_{B_h} (v_h\Div z_h + z_h \cdot \nabla v_h) \\
    &+ \gamma_{div} \| \Div z_h \|^2_{0, \Omega_h^{\Gamma}} + \gamma_{1} \|z_h + \nabla
    v_h \|^2_{0, \Omega_h^{\Gamma}} + \sigma h \left\|\left[\ddn{v_h}\right]\right\|_{0,
    \Gamma_h^i}
  \end{align*}
  A combination of (\ref{Neu:prop1}) with the Young inequality (for any
  $\varepsilon > 0$) yields  
  \begin{align}
    a^N_h (v_h, z_h ; v_h, z_h) &\geqslant (1 - \alpha) |v_h |_{1, \Omega_h}^2
		    - \frac{\varepsilon}{2} \|v_h \|^2_{0, \Omega_h^{\Gamma}}
				\notag\\
    &+ \left( \gamma_{div} - \frac{1}{2 \varepsilon} \right) \| \Div z_h \|^2_{0,
    \Omega_h^{\Gamma}} + (\gamma_{1} - \delta)  \|z_h + \nabla v_h \|^2_{0,
    \Omega_h^{\Gamma}} + (\sigma - \beta) h \left\|\left[\ddn{v_h}\right]\right\|_{0,
    \Gamma_h^i}^2
		\label{aNhest1}
  \end{align}
	
\red{
  We have for any $v_h \in \tilde{V}_h$
\begin{eqnarray}
  \|v_h \|_{0, \Omega_h}^{} \leqslant C_p^{}  |v_h |_{1, \Omega_h}^{} 
  \label{PoinRob}
\end{eqnarray}
with a mesh independent constant $C_p$. This is the standard Poincaré
inequality (recall that $\int_{\Omega_h} v_h = 0$) but one should check that
$C_p$ does not depend on $\Omega_h$. This can be easily done thanks to the
Poincaré inequality on $\Omega$. Indeed, denoting $\bar{v}_{h, \Omega} =
\frac{1}{| \Omega |} \int_{\Omega} v_h$ we have
\begin{eqnarray*}
  \| v_h \|_{0, \Omega_h} \leqslant \| v_h - \bar{v}_{h, \Omega} \|_{0,
  \Omega} + \| \bar{v}_{h, \Omega} \|_{0, \Omega} + \| v_h \|_{0, \Omega_h
  \setminus \Omega}
\end{eqnarray*}
and
\begin{eqnarray*}
  \| \bar{v}_{h, \Omega} \|_{0, \Omega} = \frac{1}{\sqrt{| \Omega |}} \left|
  \int_{\Omega} v_h \right| = \frac{1}{\sqrt{| \Omega |}} \left|
  \int_{\Omega_h \setminus \Omega} v_h \right| \leqslant \frac{\sqrt{|
  \Omega_h \setminus \Omega |}}{\sqrt{| \Omega |}} \| v_h \|_{0, \Omega_h
  \setminus \Omega}
\end{eqnarray*}
The ratio of domain measures $| \Omega_h \setminus \Omega | / | \Omega |$ is
uniformly bounded under the assumption of $h$ sufficiently small. We
have thus, by Poincaré inequality on $\Omega$ (with a constant which is
obviously mesh independent) and then by Lemma \ref{Prelim1}
\begin{eqnarray*}
  \| v_h \|_{0, \Omega_h} \leqslant C (| v_h |_{1, \Omega} + \| v_h \|_{0,
  \Omega_h \setminus \Omega}) \leqslant C \left( | v_h |_{1, \Omega} +
  \sqrt{h} \| v_h \|_{0, \partial \Omega} + h | v_h |_{1, \Omega_h \setminus
  \Omega} \right)
\end{eqnarray*}
Finally, invoking the trace inequality $\| v_h \|_{0, \partial \Omega}
\leqslant C (\| v_h \|_{0, \Omega} + | v_h |_{1, \Omega})$ and recalling that
$h$ is sufficiently small, we obtain (\ref{PoinRob}).
}

  Substituting (\ref{PoinRob}) into (\ref{aNhest1}) entails  
  \begin{align*}
    a^N_h (v_h, z_h ; v_h, z_h) & \geqslant 
		\left( 1 - \alpha -  \frac{\varepsilon}{2} C_p^2 \right) |v_h |_{1, \Omega_h}^2 
			+ \left( \gamma_{div} - \frac{1}{2 \varepsilon} \right) \| \Div z_h \|^2_{0,
    \Omega_h^{\Gamma}}\\
    & + (\gamma_{1} - \delta)  \|z_h + \nabla v_h \|^2_{0, \Omega_h^{\Gamma}} +
    (\sigma - \beta) h \left\|\left[\ddn{v_h}\right]\right\|_{0, \Gamma_h^i}
  \end{align*}
    Taking $\varepsilon,\beta$ sufficiently small and $\gamma_{1}, \gamma_{div}$ sufficiently
  big this bounds $a^N_h (v_h, z_h ; v_h, z_h)$ from below by $c \triple{v_h,
  z_h}_h^2$ as claimed.
\end{proof}

\subsection{Neumann boundary conditions: error estimates. }\label{ConvRob}

We can now establish an optimal $H^1$ error estimate for the method (\ref{Neu:Ph})--(\ref{Rob:ah}) using the coerciveness of $a^N_h$ provided by Lemma \ref{LemRob:coer}. An $L^2$ error estimate will follow in Theorem \ref{ThRob:L2}.

\begin{theorem}
  \label{ThNeu:aprio} Suppose $f \in H^1 (\Omega_h)$, $g \in H^{3 / 2}
  (\Gamma)$ and let $u \in H^3 (\Omega)$ be a solution to (\ref{Neu:P}),
  $(u_h,y_h) \in \tilde{V}_h\times Z_h$ be the solution to (\ref{Neu:Ph}). Provided $\gamma_{div}, \gamma_{1}$ are sufficiently big, there exists an $h$-independent constant $C > 0$ such that
  \begin{equation}
    |u - u_h |_{1, \Omega}  \le Ch (\|f\|_{1, \Omega_h} +\|g\|_{3 / 2, \Gamma})
    . \label{Neu:aprio}
  \end{equation}
\end{theorem}

\begin{proof}
  Under the Theorem's assumptions, the solution to (\ref{Neu:P}) is indeed in
  $H^3 (\Omega)$ and it can be extended to a function $\tilde{u} \in H^3
  (\Omega_h)$ such that $\tilde{u} = u$ on $\Omega$ and $\| \tilde{u} \|_{3,
  \Omega_h} \le C (\|f\|_{1, \Omega} +\|g\|_{3 / 2, \Gamma})$. Introduce $y =
  - \nabla \tilde{u}$ on $\Omega_h^{\Gamma}$. Clearly, $\tilde{u}$, $y$
  satisfy

  \begin{eqnarray*}
    a^N_h (\tilde{u}, y ; v_h, z_h) = \int_{\Omega_h} \tilde{f} v_h +
    \int_{\Gamma} gv_h + \gamma_{div} \int_{\Omega_h^{\Gamma}} \tilde{f}  \Div z_h
    \quad \forall (v_h, z_h) \in \tilde{V}_h \times Z_h
  \end{eqnarray*}
  with $\tilde{f} := - \Delta \tilde{u}$. It entails a Galerkin orthogonality relation 
	\begin{equation}
	 a^N_h  (\tilde{u} - u_h,y-y_h; v_h,z_h) = 
	  \int_{\Omega_h} (\tilde{f} - f) v_h
		+ \gamma_{div}\int_{\Omega_h} (\tilde{f} - f) \Div z_h, \quad \forall (v_h, z_h) \in \tilde{V}_h \times Z_h
	\label{Rob:GO}
	\end{equation} 
	We have then using the standard nodal interpolation $I_h : C (\bar{\Omega}_h) \to \tilde{V}_h$
  and recalling Lemma \ref{LemRob:coer}
  \begin{align*}
    \frac{1}{c} \triple{u_h - I_h  \tilde{u}, y_h - I_h y}_h & \le \sup_{(v_h,
    z_h) \in \tilde{V}_h \times Z_h}  \frac{a^N_h  (u_h - I_h  \tilde{u}, y_h - I_h
    y ; v_h, z_h)}{\triple{v_h, z_h}_h}\\
    & = \sup_{(v_h, z_h) \in \tilde{V}_h \times Z_h}  \frac{a^N_h (e_u, e_y ; v_h,
    z_h) + (f - \tilde{f}, v_h)_{L^2 (\Omega_h)} + \gamma_{div} (f - \tilde{f}, \Div
    z_h)_{L^2 (\Omega^{\Gamma}_h)}}{\triple{v_h, z_h}_h}\\
    & 
  \end{align*}
    with $e_u = \tilde{u} - I_h  \tilde{u}, e_y = y_{} - I_h y$. 
		Recalling (\ref{Neu:ah}), we can bound
  \begin{align*}
    a^N_h (e_u, e_y ; v_h, z_h) & \leqslant |e_u |_{1, \Omega_h} |v_h |_{1,
    \Omega_h} + \left\| \Div e_y \right\|_{0, B_h} \|v_h \|_{0, B_h} + \|e_y
    \|_{0, B_h} |v_h |_{1, B_h} 
		\\
		&+ \gamma_{div} \left\| \Div e_y \right\|_{0,
    \Omega_h^{\Gamma}} \left\| \Div z_h \right\|_{0, \Omega_h^{\Gamma}}
    + \gamma_{1} \|e_y + \nabla e_u \|_{0, \Omega_h^{\Gamma}}  \|z_h +
    \nabla v_h \|_{0, \Omega_h^{\Gamma}} + \sigma h \left\|\left[\ddn e_u\right]\right\|_{0,
    \Gamma_h^i}  \left\|\left[\ddn{v_h}\right]\right\|_{0, \Gamma_h^i}\\
    & \leqslant C\left( |e_u |^2_{1, \Omega_h} + \left\| e_u \right\|_{0,\Gamma}^2 + \left| \Div e_y \right|^2_{0,
    \Omega_h^{\Gamma}} +\|e_y \|^2_{0, \Omega_h^{\Gamma}} + h\left\|\left[\ddn{e_u}\right]\right\|_{0, \Gamma_h^i}^2 \right)^{\frac{1}{2}} \triple{v_h, z_h}_h
  \end{align*}
    By the usual interpolation estimates this entails
  \begin{eqnarray*}
    a^N_h (e_u, e_y ; v_h, z_h) \leqslant Ch (| \tilde{u} |_{2, \Omega_h} +
    |y|_{2, \Omega^{\Gamma}_h}) \triple{v_h, z_h}_h \leqslant Ch \|\tilde{u}\|_{3, \Omega_h} \triple{v_h, z_h}_h
  \end{eqnarray*}
  since $y = - \nabla \tilde{u}$ on $\Omega_h^{\Gamma}$.   Moreover,
  \begin{eqnarray*}
    (f - \tilde{f}, v_h)_{L^2 (\Omega_h)} + \gamma_{div} (f - \tilde{f}, \Div z_h)_{L^2 (\Omega^{\Gamma}_h)}
		\leqslant \|f - \tilde{f} \|_{0, \Omega_h} 
		\left(\|v_h \|_{0, \Omega_h} +\|\Div z_h \|_{0, \Omega_h^\Gamma} \right)
		\le C \|f - \tilde{f} \|_{0, \Omega_h}
    \triple{v_h,z_h}_h
  \end{eqnarray*}
	thanks to (\ref{PoinRob}). 
	We recall now that $f = \tilde{f}$ on $\Omega$ so that, thanks to Lemma \ref{Prelim1},
  \begin{equation}\label{fftil}
    \|f - \tilde{f} \|_{0, \Omega_h} = \|f - \tilde{f} \|_{0, \Omega_h^\Gamma} \le Ch |f -
    \tilde{f} |_{1, \Omega_h} \le Ch (|f|_{1, \Omega_h} +\| \tilde{u} \|_{3,
    \Omega_h}) 
  \end{equation}
  and conclude
  \begin{eqnarray*}
    \triple{u_h - I_h  \tilde{u}}_h \le C (h| \tilde{u} |_{2, \Omega_h} +\|f -
    \tilde{f} \|_{0, \Omega_h})
  \end{eqnarray*}
  Combining the estimates above with the triangle inequality proves
  $\triple{u_h - \tilde{u},y_h-y}_h \le Ch (\|f\|_{1, \Omega_h} +\|g\|_{5 / 2,
  \Gamma})$.
\end{proof}

\begin{remark}
The proof above does not rely on a solution to the non-standard boundary value problem (\ref{Neu:Pext}) on $\Omega_h$. We rather use the well defined solution $u$ to problem (\ref{Neu:P}) and extend it to $\Omega_h$. The optimal convergence is then obtained at the expense of a stronger than usual assumption on the right-hand side in (\ref{Dir:P}): we need $u\in H^3, f\in H^1$ where as $u\in H^2, f\in L^2$ suffices for standard $P_1$ finite elements on a conforming mesh.   
\end{remark}

\begin{theorem}\label{ThRob:L2}
  Under the assumptions of Theorem \ref{ThNeu:aprio}, there exists an
  $h$-independent constant $C > 0$ such that
  \begin{equation}
    \inf_{c\in\RR}\|u - u_h -c \|_{0, \Omega} \le Ch^{3 / 2}  (\|f\|_{1, \Omega_h} +\|g\|_{3/2, \Gamma}) . \label{Neu:aprioL2}
  \end{equation}
\end{theorem}

\begin{proof}
  Let us introduce $w : \Omega \to \mathbb{R}$ such that
  \begin{eqnarray*}
     - \Delta w = u - u_h-c  \text{ in } \Omega, \qquad
    \frac{\partial w}{\partial n} = 0 \text{ on } \Gamma, 
		\qquad \int_\Omega w=0
		\label{Neu:z}
  \end{eqnarray*}
	with $c\in\RR$ chosen to ensure that this problem is well posed, i.e. $c=\frac{1}{|\Omega|}\int_\Omega (u-u_h)$.
  By elliptic regularity, $\|w\|_{2, \Omega} \leq C \|u - u_h -c \|_{0, \Omega}$.
  Let $\tilde{w}$ be an extension of $w$ from $\Omega$ to $\Omega_h$
  preserving the $H^2$ norm estimate and set $w_h = I_h  \tilde{w}$. Integration by parts and interpolation estimates yield
	\begin{align*}
  \|u - u_h -c\|_{0, \Omega}^2 & = \int_{\Omega} \nabla (u - u_h) \cdot \nabla
  (w - w_h) + \int_{\Omega} \nabla (u - u_h) \cdot \nabla w_h \\
  & \le Ch | \tilde{u} - u_h |_{1, \Omega_h} |\tilde w|_{2, \Omega_h}
	+ \left| \int_{\Omega}\nabla (u - u_h) \cdot \nabla w_h  \right|
\end{align*}
	We now rewrite the bilinear form (\ref{Neu:ah}) as
\begin{align*}
  a^N_h (u, y ; v, z) & = \int_{\Omega} \nabla u \cdot \nabla v 
	+ \int_{B_h} (v\Div{y} + (y + \nabla u) \cdot \nabla  v) + \gamma_{div} \int_{\Omega_h^{\Gamma}} \Div y \Div z\\
  & + \gamma_{1} \int_{\Omega_h^{\Gamma}} (y + \nabla u) \cdot (z + \nabla v) +
  \sigma h \int_{\Gamma_h^i} \left[ \frac{\partial u}{\partial n} \right]
  \cdot \left[ \frac{\partial v}{\partial n} \right]
\end{align*}	
so that the Galerkin orthogonality relation (\ref{Rob:GO}) with $v_h=w_h-\bar{w}_{h,\Omega_h}, z_h=0$ and $\bar{w}_{h,\Omega_h}$ the average of $w_h$ over $\Omega_h$ becomes 
  \begin{align*}
      \int_{\Omega} \nabla (\tilde{u} - u_h) \cdot \nabla w_h 
			&			+ \int_{B_h}  (\Div (y - y_h) (w_h-\bar{w}_{h,\Omega_h}) + (y + \nabla \tilde{u} - y_h - \nabla u_h) \cdot
      \nabla w_h)\\
      & + \gamma_{1} \int_{\Omega_h^{\Gamma}} (y + \nabla \tilde{u} - y_h - \nabla u_h) \cdot \nabla w_h  
			+ \sigma h \int_{\Gamma_h^i} \left[ \frac{\partial (\tilde{u} -u_h)}{\partial n} \right] \cdot \left[ \frac{\partial w_h}{\partial n} \right]\\
      & = \int_{\Omega_h} (\tilde{f} - f) (w_h-\bar{w}_{h,\Omega_h})
  \end{align*}
  Recalling the definition of the triple norm from Lemma \ref{LemRob:coer}, this leads to
  \begin{align*}
    \left| \int_{\Omega} \nabla (\tilde{u} - u_h) \cdot \nabla w_h 
		 \right|
     &\leqslant C \triple{\tilde{u} - I_h  \tilde{u}, y - I_h y}_h 
		\left( \|w_h-\bar{w}_{h,\Omega_h}\|_{0, \Omega_h^{\Gamma}} + |w_h|_{1, \Omega_h^{\Gamma}} + \sqrt{h} \left\|\left[\ddn{w_h}\right]\right\|_{0,
    \Gamma_h^i} \right) \\
		&+ \| \tilde{f} - f \|_{0, \Omega_h^{\Gamma}} \| w_h-\bar{w}_{h,\Omega_h}\|_{0, \Omega_h^{\Gamma}}
  \end{align*}
  By Lemma \ref{Prelim1} and interpolation estimates
  \begin{eqnarray*}
    \| w_h \|_{0, \Omega_h^{\Gamma}} \leqslant \| \tilde{w} - I_h \tilde{w}
    \|_{0, \Omega_h^{\Gamma}} + \| \tilde{w} \|_{0, \Omega_h^{\Gamma}}
    \leqslant {Ch}^2 | \tilde{w} |_{2, \Omega_h^{\Gamma}} + C \left(
    \sqrt{h} \| \tilde{w} \|_{0, \Gamma} + h| \tilde{w} |_{1,
    \Omega_h^{\Gamma}} \right) \leqslant C \sqrt{h} \| \tilde{w} \|_{2,
    \Omega_h}
  \end{eqnarray*}
  and similarly
  \begin{eqnarray*}
    \| \nabla w_h \|_{0, \Omega_h^{\Gamma}} \leqslant \| \nabla (\tilde{w} -
    I_h \tilde{w}) \|_{0, \Omega_h^{\Gamma}} + \| \nabla \tilde{w} \|_{0,
    \Omega_h^{\Gamma}} \leqslant {Ch}^{} | \tilde{w} |_{2,
    \Omega_h^{\Gamma}} + C \left( \sqrt{h} \| \nabla \tilde{w} \|_{0, \Gamma}
    + h| \nabla \tilde{w} |_{1, \Omega_h^{\Gamma}} \right) \leqslant C
    \sqrt{h} \| \tilde{w} \|_{2, \Omega_h}
  \end{eqnarray*}
	Finally,
  \begin{multline*}
    \| \bar{w}_{h, \Omega_h} \|_{0, \Omega_h^{\Gamma}} = \frac{\sqrt{|
    \Omega_h^{\Gamma} |}}{| \Omega_h |} \left| \int_{\Omega_h} w_h \right| =
    \frac{\sqrt{| \Omega_h^{\Gamma} |}}{| \Omega_h |} \left| \int_{\Omega_h}
    I_h \tilde{w} - \int_{\Omega} w \right| = \frac{\sqrt{| \Omega_h^{\Gamma}
    |}}{| \Omega_h |} \left| \int_{\Omega_h} (I_h \tilde{w} - \tilde{w}) +
    \int_{\Omega_h \setminus \Omega} \tilde{w} \right| \\ \leqslant \frac{\sqrt{|
    \Omega_h^{\Gamma} |}}{\sqrt{| \Omega_h |}} \| \tilde{w} - I_h \tilde{w}
    \|_{0, \Omega_h} + \frac{| \Omega_h^{\Gamma} |}{| \Omega_h |} \| \tilde{w}
    \|_{0, \Omega^{\Gamma}_h} \leqslant Ch \| \tilde{w} \|_{2, \Omega_h}
  \end{multline*}
  
	Combining all the estimates above we arrive at
	$$
	  \|u - u_h -c\|_{0, \Omega}^2 \le C\sqrt{h}
		(\triple{\tilde{u} - I_h  \tilde{u}, y_{} - I_h y}_h + \| \tilde{f} - f \|_{0, \Omega_h^{\Gamma}}) 
		  \|\tilde w\|_{2, \Omega_h}
  $$
  We conclude recalling the bound on the triple norm of the error from  the proof of Theorem \ref{ThNeu:aprio}, (\ref{fftil}), and the regularity estimate $\|\tilde w\|_{2, \Omega_h} \le C\|u - u_h -c \|_{0, \Omega}$.
\end{proof}

\section{Extension to $P_k$ finite elements \red{and Robin boundary conditions}}
The methods presented above can be easily extended to $P_k$ finite elements giving optimal convergence of order $h^k$ in the $H^1$-norm, $k\ge 2$. We thus consider in this section the finite element space
\begin{equation}\label{Vhk}
  V_h^{(k)} = \{v_h \in H^1 (\Omega_h) : v_h |_T \in \mathbb{P}_k (T),
  \quad \forall T \in \mathcal{T}_h \}
\end{equation}
with $\mathbb{P}_k (T)$ representing the polynomials of degree $\le k$,
propose some modifications to be introduced to the methods above, first for Dirichlet boundary conditions and then for Neumann-Neumann ones, and outline the convergence proofs in both these cases. 

\subsection{Dirichlet boundary conditions}
The method (\ref{Dir:Ph})--(\ref{Dir:ah}) is modified as follows: \\
Find $u_h \in V_h^{(k)}$ s.t. 
\[ a_h (u_h, v_h) = \int_{\Omega_h} fv_h + \int_{\Gamma} g \frac{\partial
   v_h}{\partial n} + \frac{\gamma}{h}  \int_{\Gamma} gv_h 
	{- \sigma h^2\sum_{T \subset \Th^{\Gamma}} \int_T f \Delta v_h}
	,\quad \forall v_h \in V_h^{(k)}	
\]
with
\begin{align*}
  a_h (u, v) &= \int_{\Omega_h} \nabla u \cdot \nabla v - \int_{\Gamma_h}
  \frac{\partial u}{\partial n} v + \int_{\Gamma} u \frac{\partial v}{\partial
  n} + \frac{\gamma}{h}  \int_{\Gamma} uv \\
	&\quad {+ \sigma h^2 \sum_{T \subset\Th^{\Gamma}} \int_T (\Delta u) (\Delta v)}
  + \sigma \sum_{E \in \mathcal{F}_{\Gamma}} \sum_{j=1}^k
	h^{2j-1} \int_E \left[
  \frac{\partial^j u}{\partial n^j} \right] \left[ \frac{\partial^j v}{\partial n^j}
  \right]
\end{align*}
Note that the additional stabilization term with the product  $(\Delta u) (\Delta v)$ is strongly consistent since $-\Delta u=f$. Also note  that the ghost penalty term is extended to control the normal derivatives of all orders up to $k$, cf. \cite{burman3}.

Revisiting the theoretical analysis of Section 3 reveals the following:
\begin{itemize}
	\item Lemma \ref{LemDir:prop1} remains valid thanks to the extended ghost penalty. Indeed, the only thing to recheck in its proof is the following implication: if $\nabla v_h=0$ on $T_k$ and all the norms of the jumps contained in ghost penalty vanish, then $\nabla v_h=0$ on $\Pi_k^\Gamma$. This is true since the extended ghost penalty controls all the derivatives present in our finite element space.
	\item The first relation (\ref{LemDir:eq1}) in the proof of Lemma \ref{LemDir:coer} now becomes
	\[ 
	 \int_{\Gamma_h} \frac{\partial v_h}{\partial n} v_h - \int_{\Gamma} 
   \frac{\partial v_h}{\partial n} v_h = \cdots 
	 = \int_{B_h} |\nabla v_h |^2 + \sum_{T \in \Th^{\Gamma}} \int_{T\cap B_h} (\Delta v_h)v_h
	  - \sum_{F \in \mathcal{F}_{\Gamma}} \int_{F \cap B_h} v_h
   \left[ \frac{\partial v_h}{\partial n} \right] 
	\] 
since one can no longer assume $\Delta v_h=0$  on $T$. The new term with $\Delta v_h$ can be then controlled thanks to the additional $\Delta\cdot\Delta$  stabilization term in the bilinear form $a_h$ so that Lemma \ref{LemDir:coer} remains valid.    
  \item In Theorem \ref{ThDir:aprio}, we should now suppose $f \in H^{k-1} (\Omega_h)$, $g \in H^{k+1/2}
  (\Gamma)$ so that the solution $u$ to (\ref{Dir:P}) is in $H^{k+1} (\Omega)$. The proof of the Theorem can be then followed using appropriate Sobolev spaces and standard interpolation estimates to $P_k$ finite elements. The only non-trivial change in the proof concerns the estimate for $\|f - \tilde{f} \|_{0, \Omega_h^\Gamma}$, which should be changed to	
  \begin{equation}\label{fftilk}
     \|f - \tilde{f} \|_{0, \Omega_h^\Gamma} \le Ch^{k-1} |f - \tilde{f} |_{k-1, \Omega_h} \le Ch^{k-1} (|f|_{k-1, \Omega_h} +\| \tilde{u} \|_{k+1,    \Omega_h}) .
  \end{equation}
This can be proved by refining the argument of 	Lemma \ref{Prelim1}. We recall that $f = \tilde{f}$ on $\Omega$ and both $f$ and $\tilde{f}$are in $H^{k-1}(\Omega)$ so that all the derivatives up to order $k-2$ of $f$ and $\tilde{f}$ coincide on $\Gamma$. The bound (\ref{fftilk}) then follows as in Lemma \ref{Prelim1} employing a Taylor expansion with the integral remainder of order $k-1$.
		\end{itemize}
Keeping in mind the modifications above, we can easily establish the convergence of the $P_k$ version of our method: given $f \in \red{H^{k-1} (\Omega_h)}$ and $g \in \red{H^{k+1/2} (\Gamma)}$, and supposing $\sigma$ sufficiently big, one has
  \begin{equation*}
    |u - u_h |_{1, \Omega} + \frac{1}{\sqrt{h}}  \|u - u_h \|_{0, \Gamma} \le
    Ch^k (\|f\|_{k-1, \Omega_h} +\|g\|_{k+1/2, \Gamma}) . 
  \end{equation*}
The $L^2$ error estimate of order $h^{k+1/2}$ can be also proved as in Theorem \ref{ThDir:L2}.
		
\subsection{Neumann boundary conditions}
We turn now to Problem (\ref{Neu:P}). The goal is to extend the method (\ref{Neu:Ph})--(\ref{Rob:ah}) to $P_k$ finite elements. We thus recall the space $V_h^{(k)}$ as in (\ref{Vhk}), introduce its subspace $\tilde{V}_h^{(k)}$ as in (\ref{Vhtild}) and the auxiliary finite element space
\begin{eqnarray*}
  Z_h^{(k)} = \{z_h \in H^1 (\Omega^{\Gamma}_h)^d : z_h |_T \in \mathbb{P}_k (T)^d,
  \quad  \forall T \in \mathcal{T}_h^{\Gamma} \}
\end{eqnarray*}
Our finite element problem is: Find $u_h \in \tilde{V}_h^{(k)}$, $y_h \in Z_h^{(k)}$ solving
(\ref{Neu:Ph}) with the bilinear form $a^N_h$ modified as follows
\begin{align*}
  a^N_h (u, y ; v, z) &= \int_{\Omega_h} \nabla u \cdot \nabla v +
  \int_{\Gamma_h} y \cdot nv - \int_{\Gamma} y \cdot nv 
	+ \gamma_{div}\int_{\Omega_h^{\Gamma}} \Div y \Div z \\
   & + \gamma_1 \int_{\Omega_h^{\Gamma}} (y + \nabla u) \cdot (z + \nabla  v) 
	+ \sigma \sum_{j=1}^k	h^{2j-1} \int_{\Gamma_h^i} \left[
  \frac{\partial^j u}{\partial n^j} \right] \left[ \frac{\partial^j v}{\partial n^j} \right]
	+ \sigma \sum_{E \in \mathcal{F}^{cut}_{\Gamma}} \sum_{j=1}^{k-1}	h^{2j+1}  \left[
  \frac{\partial^j y}{\partial n^j} \right] \cdot \left[ \frac{\partial^j z}{\partial n^j} \right]
\end{align*}
We have added here extra penalization terms controlling the jumps of the higher normal derivatives of $u$ on $\Gamma_h^i$ and an additional ghost stabilization term controlling the jumps of derivatives of $y$ on the cut facets, denoted by $\mathcal{F}^{cut}_{\Gamma}$ ($\mathcal{F}^{cut}_{\Gamma}$ is thus the collection of interior facets of mesh $\Th^{\Gamma}$; $\mathcal{F}^{cut}_{\Gamma}\subset\mathcal{F}_{\Gamma}$). 

Keeping in mind the modifications above, we can easily establish the convergence of the $P_k$ version of our method: given $f \in H^k (\Omega_h)$ and $g \in H^{k+1 / 2} (\Gamma)$, ans supposing $\gamma_{div},\gamma_{1}$ sufficiently big, one has
  \begin{equation*}
    |u - u_h |_{1, \Omega} \le Ch^k (\|f\|_{k, \Omega_h} +\|g\|_{k+1/2, \Gamma}) . 
  \end{equation*}
The $L^2$ error estimate of order $h^{k+1/2}$ can be also proved as in Theorem \ref{ThDir:L2}. Note that we still need to assume some extra regularity: $f \in H^k (\Omega_h)$ contrary to $f \in H^{k-1} (\Omega_h)$ in the case of Dirichlet boundary conditions, or to what whould be necessary for the optimal convergence of the standard finite element method.

\color{black}
\subsection{Robin boundary conditions}
We can also consider the problem
\begin{equation}
  - \Delta u = f \text{ in } \Omega, \qquad u+\kappa\frac{\partial u}{\partial n}= g
  \text{ on } \Gamma \label{Rob:P}
\end{equation}
with $\kappa>0$. It is straightforward to adapt the method of the Neumann case (\ref{Neu:Ph})--(\ref{Rob:ah}) to the present Robin case. To this end, we rewrite the boundary conditions in (\ref{Rob:P}) as $\frac{1}{\kappa}u+\frac{\partial u}{\partial n}= \frac{1}{\kappa}g$. By the same considerations as in the Neumann case we arrive then at the following method for (\ref{Rob:P}) using $P_k$ finite elements (without the constraint of zero integral over $\Omega_h$): 
Find $u_h \in V_h^{(k)}$, $y_h \in Z_h^{(k)}$ such that
\begin{align}\label{Robin1}
    \int_{\Omega_h} \nabla u_h \cdot \nabla v_h &+
  \int_{\Gamma_h} y_h \cdot nv - \int_{\Gamma} y_h \cdot nv + \int_{\Gamma} \frac{1}{\kappa} u_h v_h 
	+ \gamma_{div}\int_{\Omega_h^{\Gamma}} \Div y_h \Div z_h \\
	\notag
   & + \gamma_1 \int_{\Omega_h^{\Gamma}} (y_h + \nabla u_h) \cdot (z_h + \nabla  v_h) 
	+ \sigma \sum_{j=1}^k	h^{2j-1} \int_{\Gamma_h	^i} \left[
  \frac{\partial^j u_h}{\partial n^j} \right] \left[ \frac{\partial^j v_h}{\partial n^j} \right]
	+ \sigma \sum_{E \in \mathcal{F}^{cut}_{\Gamma}} \sum_{j=1}^{k-1}	h^{2j+1}  \left[
  \frac{\partial^j y_h}{\partial n^j} \right] \cdot \left[ \frac{\partial^j z_h}{\partial n^j} \right] \\
	\notag
	&= \int_{\Omega_h} fv_h +
  \int_{\Gamma} \frac{1}{\kappa}gv_h + \gamma_{div} \int_{\Omega_h^{\Gamma}} f \Div z_h 
	\hspace{20mm}
  \forall (v_h, z_h) \in {V}_h^{(k)} \times Z_h^{(k)}
\end{align}
The proofs above can be easily adapted to show the optimal convergence of this scheme both in $H^1(\Omega)$ and $L^2(\Omega)$ norms. 

Alternatively, one can adapt the method of the Dirichlet case (\ref{Dir:Ph})--(\ref{Dir:ah}) mimicking the approach of \cite{juntunen09}. In the context of unfitted meshes ($\Gamma\not=\Gamma_h$), choosing the antisymmetric variant of the Nitsche method and adding the ghost penalty, this would give the scheme\footnote{This approach has been suggested by an anonymous reviewer.}:  
Find $u_h \in V_h^{(k)}$ such that
\begin{align}\label{Robin2}
   \int_{\Omega_h} \nabla u_h \cdot \nabla v_h &- \int_{\Gamma_h} \frac{\partial
     u_h}{\partial n} v_h + \int_{\Gamma} \left( u_h + \kappa \frac{\partial
     u_h}{\partial n} \right) \frac{\partial v_h}{\partial n} + \frac{1}{\kappa +
     h / \gamma} \int_{\Gamma} \left( u_h + \kappa \frac{\partial u_h}{\partial n}
     \right) \left( v_h + \kappa \frac{\partial v_h}{\partial n} \right) \\
	\notag
	&\quad {+ \sigma h^2 \sum_{T \subset\Th^{\Gamma}} \int_T (\Delta u_h) (\Delta v_h)}
  + \sigma \sum_{E \in \mathcal{F}_{\Gamma}} \sum_{j=1}^k
	h^{2j-1} \int_E \left[
  \frac{\partial^j u_h}{\partial n^j} \right] \left[ \frac{\partial^j v_h}{\partial n^j}
  \right] \\
		\notag
   &= \int_{\Omega_h} fv + \int_{\Gamma} g \frac{\partial v_h}{\partial n} +
     \frac{1}{\kappa + h / \gamma} \int_{\Gamma} g \left( v_h + \kappa
     \frac{\partial v_h}{\partial n} \right)
	\hspace{20mm}
  \forall v_h \in {V}_h^{(k)}
\end{align}
The advantage of (\ref{Robin2}) over (\ref{Robin1}) is in the absence of the additional variable $y_h$. However, the analysis of this scheme is yet to be done and is out of the scope of the present paper.

\color{black}

\section{Numerical experiments.} We shall illustrate our methods (\ref{Dir:Ph})--(\ref{Dir:ah}) and (\ref{Neu:Ph})--(\ref{Rob:ah}) by numerical experiments in a 2D domain $\Omega$ defined by a level-set function $\varphi$:
\begin{equation}
\Omega=\{(x,y):\varphi(x,y)<0\}\text{ with } \varphi:=r^4(5+3\sin(7\theta+7\pi/36))/2-R^4
\label{OmegaPhi}
\end{equation}
where $(r,\theta)$ are the polar coordinates $r=\sqrt{x^2+y^2}$, $\theta=\arctan\frac{y}{x}$, and $R=0.47$ (this example is taken from \cite{HaslingerRenard}). To construct the computational mesh, we embed $\Omega$ into the square $\mathcal{O}=(-0.5,0.5)^2$, introduce a regular $N\times N$ criss-cross mesh on $\mathcal{O}$, and drop the triangles outside $\Omega$ to produce $\Th$ and $\Th^\Gamma$, as illustrated at Fig.~\ref{Dir:figdom} in the case $N=16$. In some of our experiments, the domain $\Omega$ will be rotated by an angle $\theta_0$ counter-clockwise around the origin. This is achived by redefining   $\theta$ as $\theta=\arctan\frac{y}{x}-\theta_0$.
 
All the computations are done in FreeFEM \cite{freefem} taking advantage of its level-set capabilities (the key word \verb?levelset? in numerical integration commands \verb?int1d? and \verb?int2d? to deal with the integrals over $\Gamma$ and over $\Omega$, cf. Subsection \ref{CompCutFEM}). Note that the numerical integration on $\Gamma$ introduces an additional error which is not covered by our theoretical analysis: $\Gamma$ is in fact approximated by a sequence of straight segments with the endpoints obtained by approximate intersections of $\Gamma$ with the edges of $\Th$. This error is of order $h^2$ and thus can be assumed negligible in the case of $P_1$ finite elements (which is the only case studied numerically in this paper). We presume that a subtler approximation should be employed when dealing with higher order finite elements, as in \cite{burmanBVC}.

\subsection{Dirichlet boundary conditions}
We start by solving numerically the Dirichlet problem (\ref{Dir:P}) in domain (\ref{OmegaPhi}) with zero right-hand side $f=0$ and a non-homogeneous boundary condition $g$ set so that the exact solution is given by 
\begin{equation}
u=\sin(x) e^y
\label{exact}
\end{equation} 
We employ the method (\ref{Dir:Ph})--(\ref{Dir:ah}) taking the following parameter values: $\gamma=1$, $\sigma=0.01$. Fig.~\ref{Dir:figN16} represents the solution obtained on a $16\times 16$ mesh. We observe that the numerical method captures well the exact solution, the perturbation being essentially concentrated in the narrow fictitious domain $B_h=\Omega_h\setminus\Omega$. Fig.~\ref{Dir:convh} reports the evolution of the error under the mesh refinement (always using the regular criss-cross meshes as mentioned above) and confirms the optimal convergence order of the method in both $H^1$ and $L^2$ norms.
\begin{figure}[htp]
\centerline{
\includegraphics[width=0.5\textwidth]{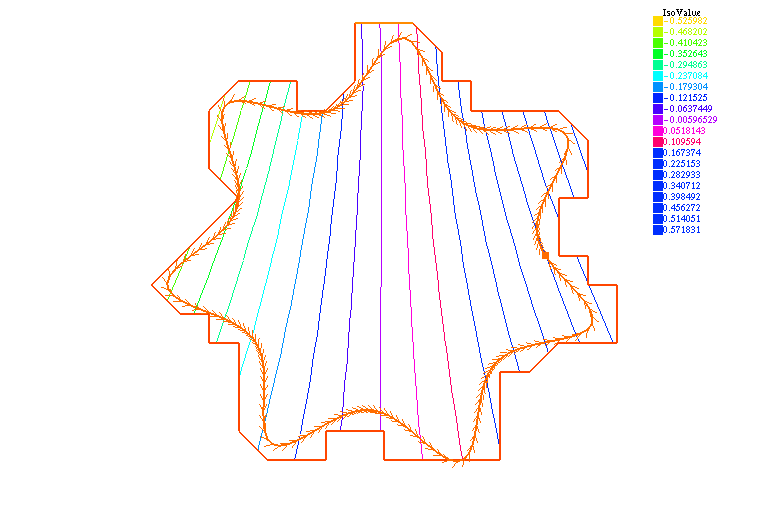}
\includegraphics[width=0.5\textwidth]{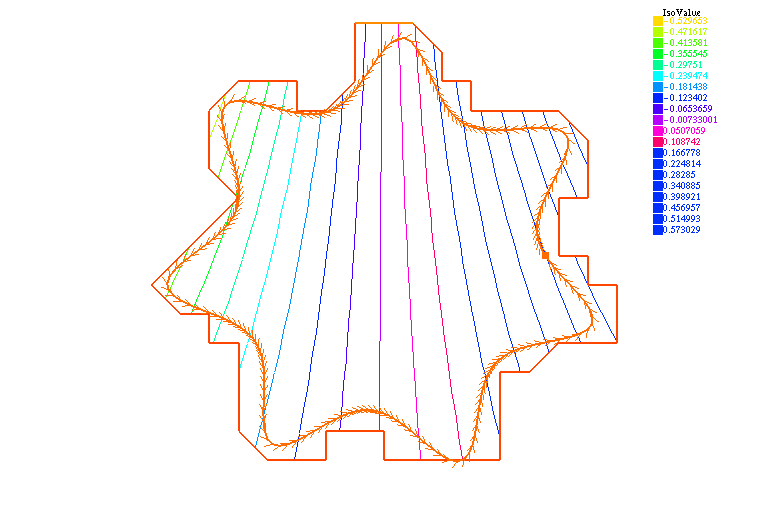}
}
\vspace{-5mm}
\centerline{
\hspace{-18mm}
\includegraphics[width=0.5\textwidth]{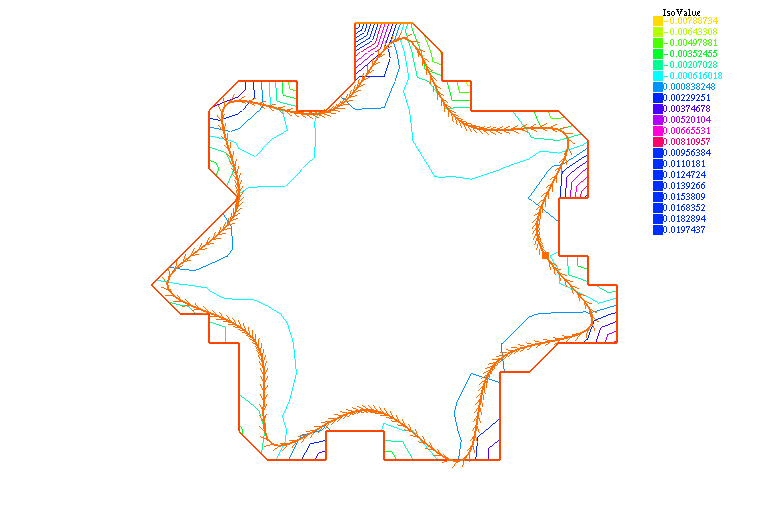}
}
\caption{Poisson-Dirichlet problem (\ref{Dir:P}) with the exact solution (\ref{exact}) on domain (\ref{OmegaPhi}). Top-Left: the numerical solution $u_h$ on the $16\times 16$ mesh as in Fig.~\ref{Dir:figdom} produced by method (\ref{Dir:Ph})--(\ref{Dir:ah}); Top-Right: the exact solution $u$; Bottom: the error $u-u_h$.}
\label{Dir:figN16}
\end{figure}
\begin{figure}[htp]
\centerline{
\includegraphics[width=0.4\textwidth]{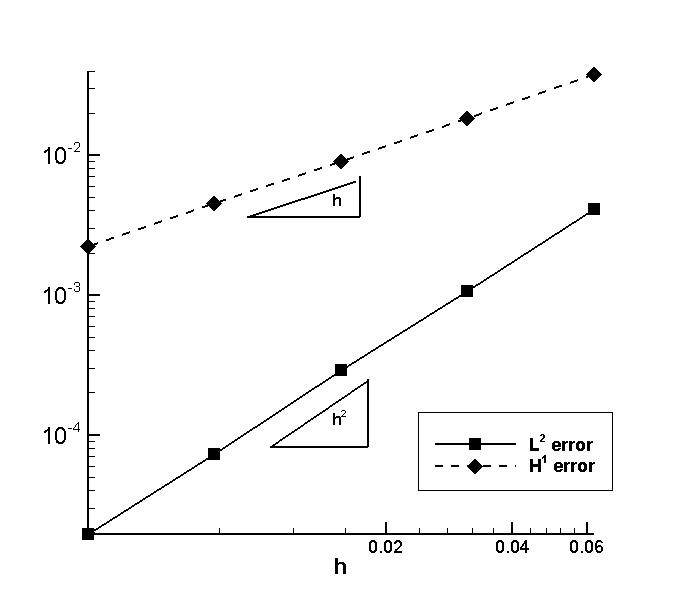}
}
\caption{Poisson-Dirichlet problem as in Fig.~\ref{Dir:figN16}: the relative errors in $L^2(\Omega)$ and $H^1(\Omega)$ norms as functions of 
$h$ under the mesh refinement.}
\label{Dir:convh}
\end{figure}
\begin{figure}[htp]
\centerline{
\includegraphics[width=0.4\textwidth]{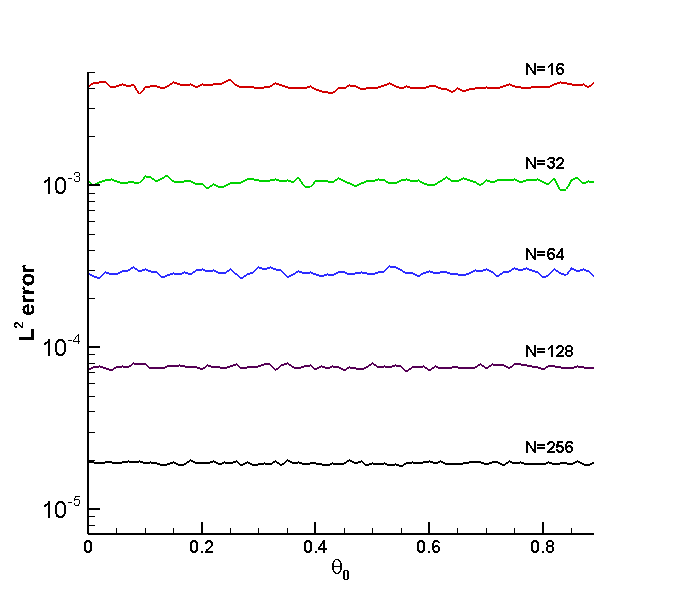}
\quad
\includegraphics[width=0.4\textwidth]{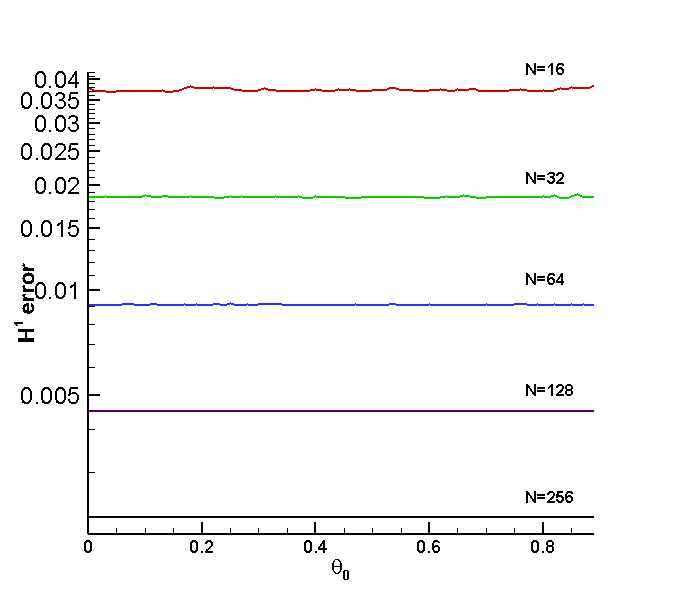}
}
\caption{Poisson-Dirichlet problem (\ref{Dir:P}) on the domain as in Figs. \ref{Dir:figdom} and \ref{Dir:figN16} rotated by an angle $\theta_0$ around the origin: the relative errors in $L^2(\Omega)$ and $H^1(\Omega)$ norms as 
functions of the rotation angle $\theta_0$.}
\label{Dir:figres}
\end{figure}
\begin{figure}[htp]
\centerline{
\includegraphics[width=0.4\textwidth]{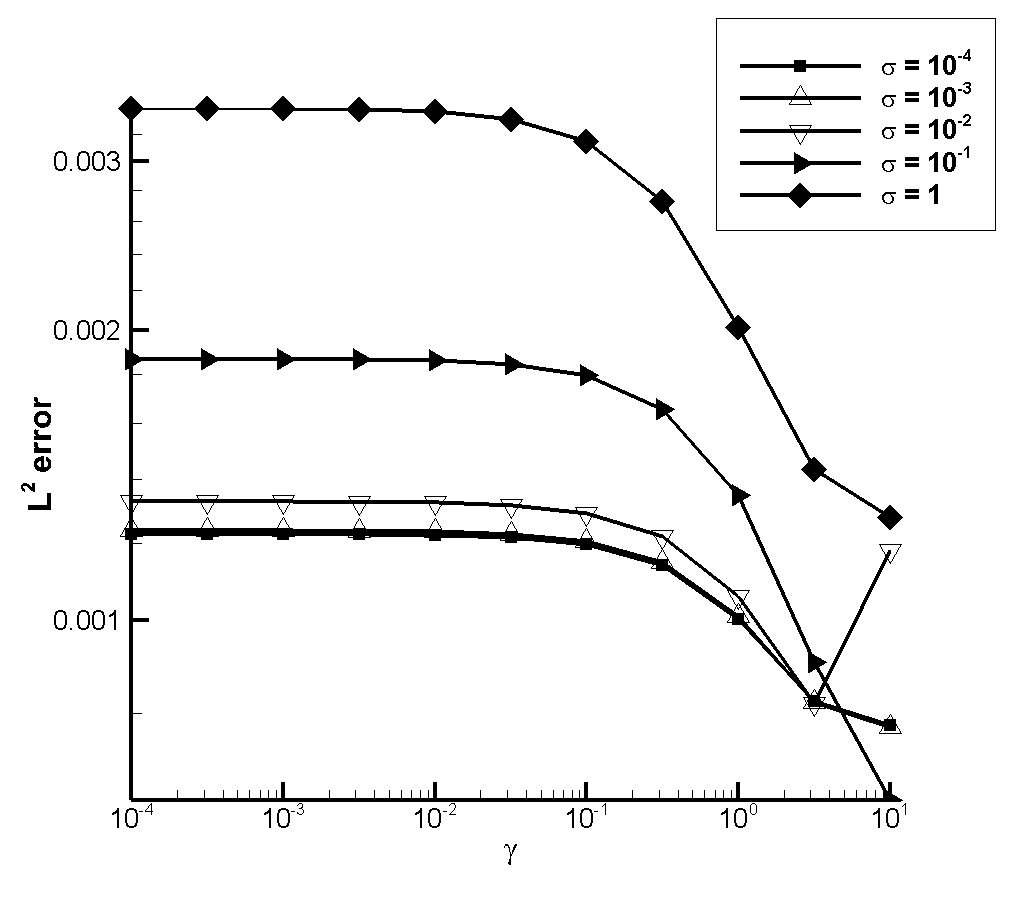}
\quad
\includegraphics[width=0.4\textwidth]{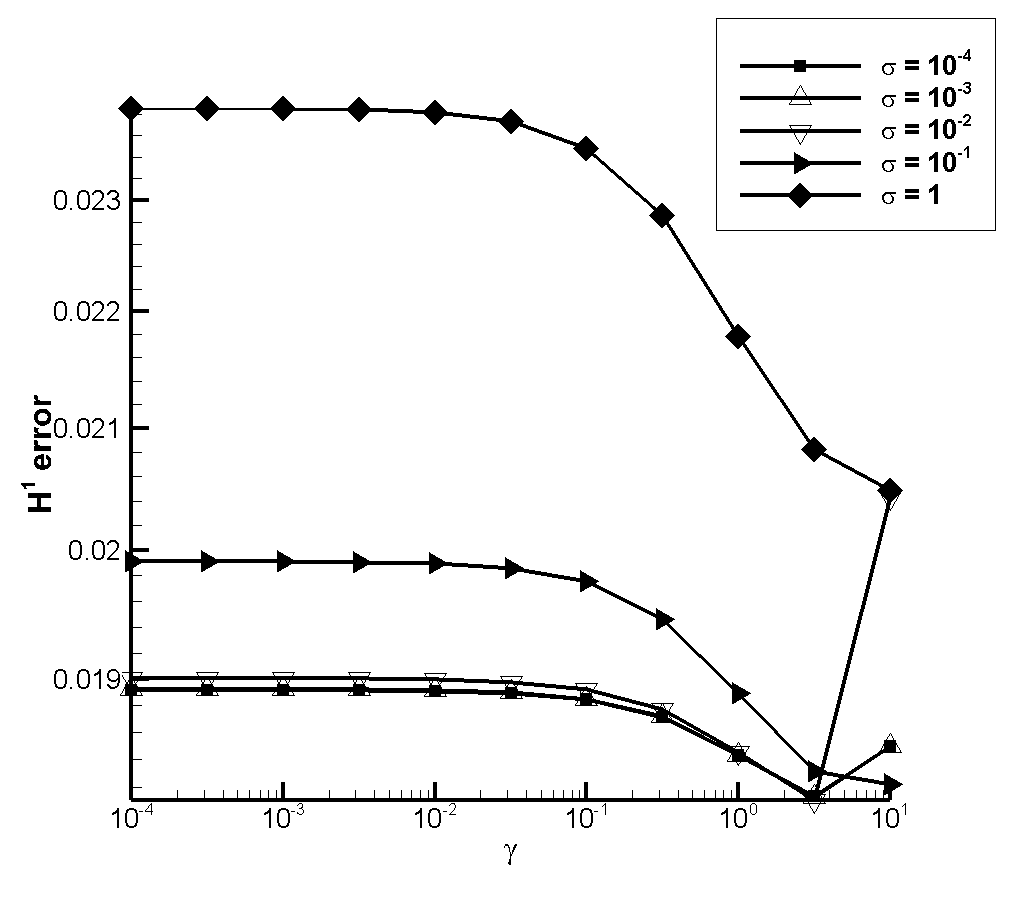}
}
\caption{Poisson-Dirichlet problem (\ref{Dir:P}) as in Fig.~\ref{Dir:figN16} but on the $32\times 32$ mesh. The relative errors in $L^2(\Omega)$ and $H^1(\Omega)$ norms as functions of parameters $\gamma$ and $\sigma$.}
\label{Dir:params}
\end{figure}

In order to explore the robustness of the method with respect to the placement of the physical domain on the computational mesh, we now redo the calculations above, rotating $\Omega$ by a series of angles $\theta_0$ ranging from $0$ to $\frac{2\pi}{7}$ as described in the preamble of this Section. For each rotation angle $\theta_0$, the boundary $\Gamma$ cuts the triangles of the background mesh in a different manner, creating sometimes the ``dangerous" situations when certain mesh triangles of $\Th$ have only a tiny portion inside the physical domain $\Omega$. Fig.~\ref{Dir:figres} presents the errors in $L^2(\Omega)$ and $H^1(\Omega)$ norms as functions of $\theta_0$ at different discretization levels. We observe that the errors do not vary much from one position to another, especially when measured in the $H^1(\Omega)$ norm. Moreover, this variability decreases as the meshes are refined.

Finally, we explore the influence of the parameters $\gamma$ and $\sigma$ in (\ref{Dir:ah})  on the precision of the method. The error norms on a fixed $32\times 32$ mesh for various choices of $\gamma$ and $\sigma$  are presented at Fig.~\ref{Dir:params}.	We observe that the method is not very sensitive to the parameters, especially when the error is measured in the $H^1$ norm. The accuracy does not seem to deteriorate catastrophically even in the limit $\gamma,\sigma\to 0$. This is somewhat surprising in view of the theoretical analysis of Section \ref{theory} and may indicate that a subtler theory could reveal the optimal convergence properties of the method (\ref{Dir:Ph})--(\ref{Dir:ah}) without stabilization. In our numerical experiments above, we have preferred however to remain on a safer side and have chosen a rather large value of the Nitsche parameter $\gamma=1$ while keeping the ghost stabilization parameter $\sigma=10^{-2}$ small.  Note that larger values of $\gamma$ seem to make the method more sensible and unpredictive with respect to the choice of $\sigma$. On the other hand, the error is clearly monotonically increasing with $\sigma$ in the regime $\gamma\le 1$. 

\subsection{Comparisons with CutFEM}\label{CompCutFEM}

\begin{figure}[htp]
\begin{subfigure}[t]{\textwidth}
\centering
\includegraphics[width=0.4\textwidth]{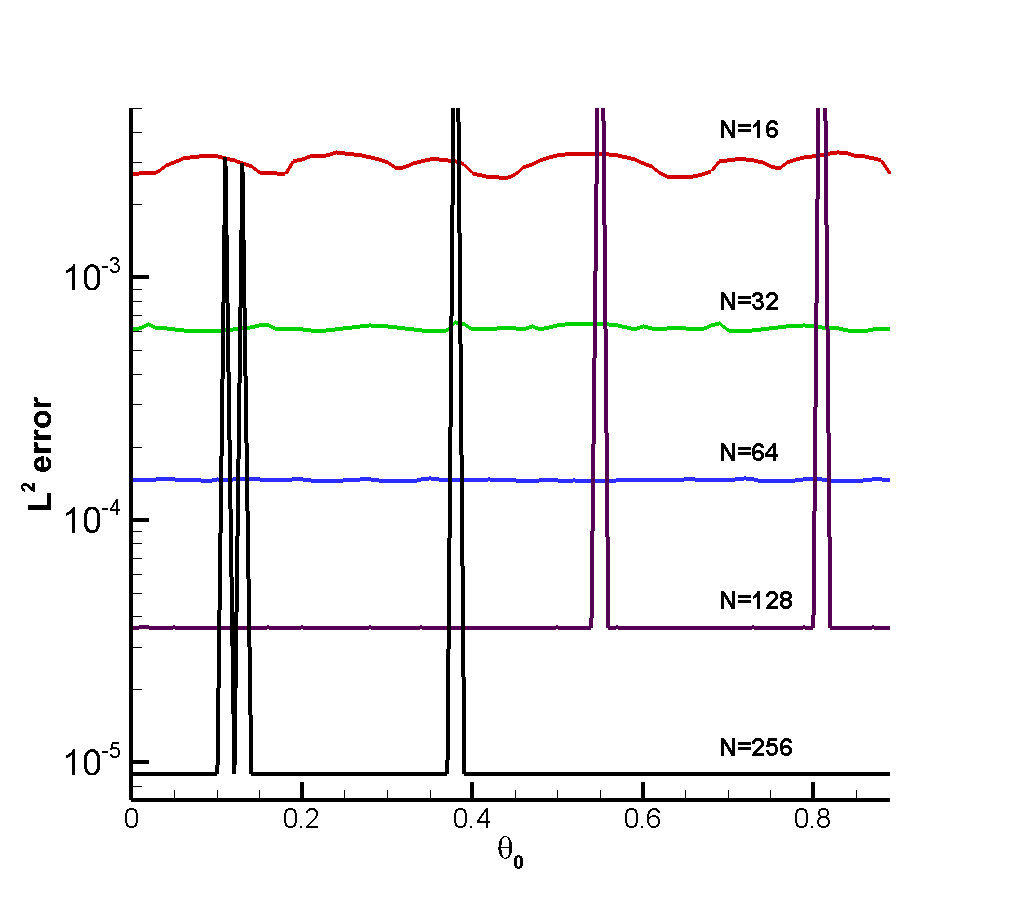}
\quad
\includegraphics[width=0.4\textwidth]{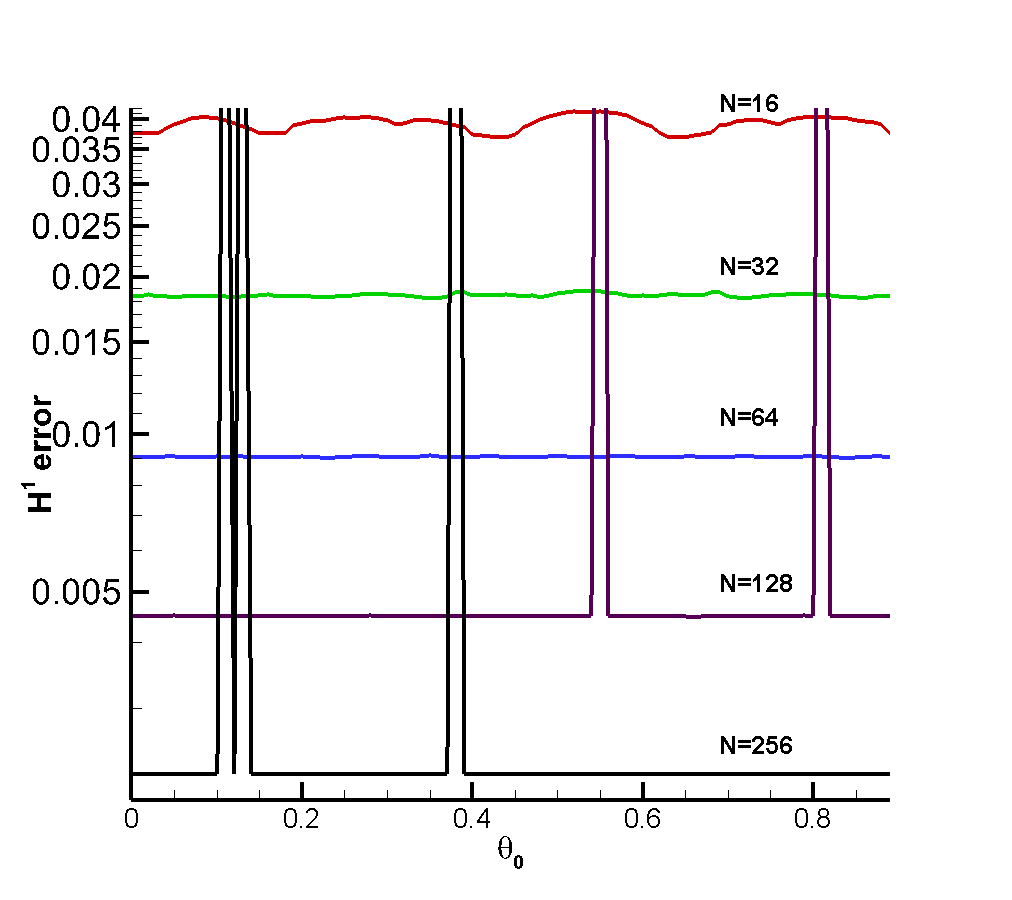}
\caption{CutFEM with Lagrange multipliers (\ref{CutLam}) with $\sigma=0.01$}
\end{subfigure}

\begin{subfigure}[t]{\textwidth}
\centering
\includegraphics[width=0.4\textwidth]{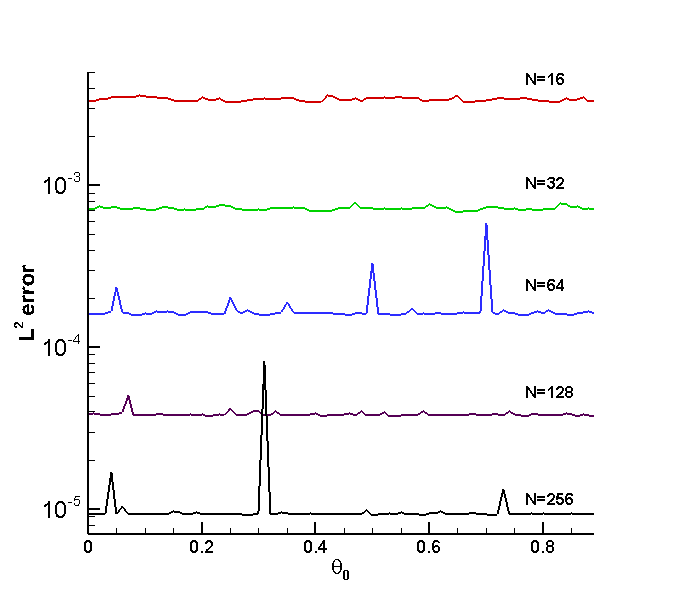}
\quad
\includegraphics[width=0.4\textwidth]{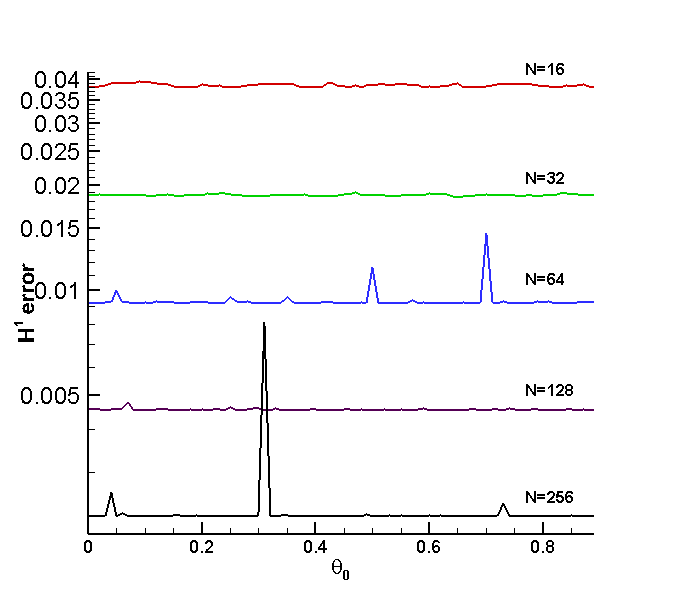}
\caption{Symmetric Nitsche CutFEM (\ref{CutSym}) with $\gamma=5$, $\sigma=0.1$}
\end{subfigure}

\begin{subfigure}[t]{\textwidth}
\centering
\includegraphics[width=0.4\textwidth]{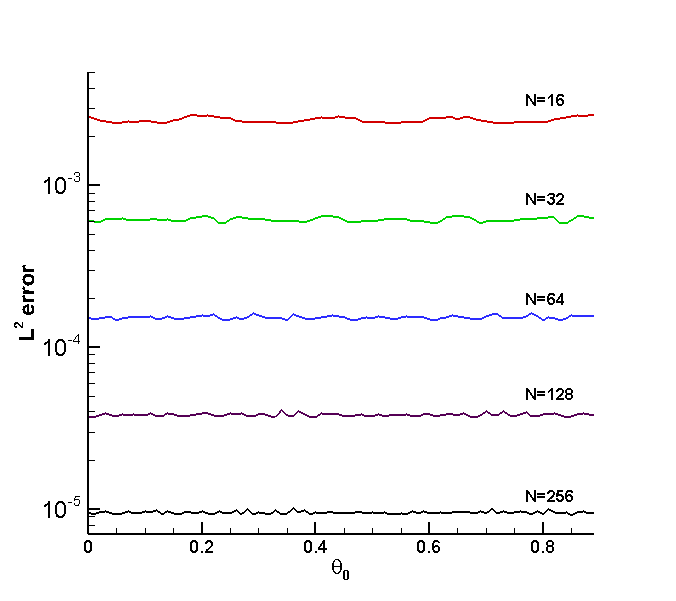}
\quad
\includegraphics[width=0.4\textwidth]{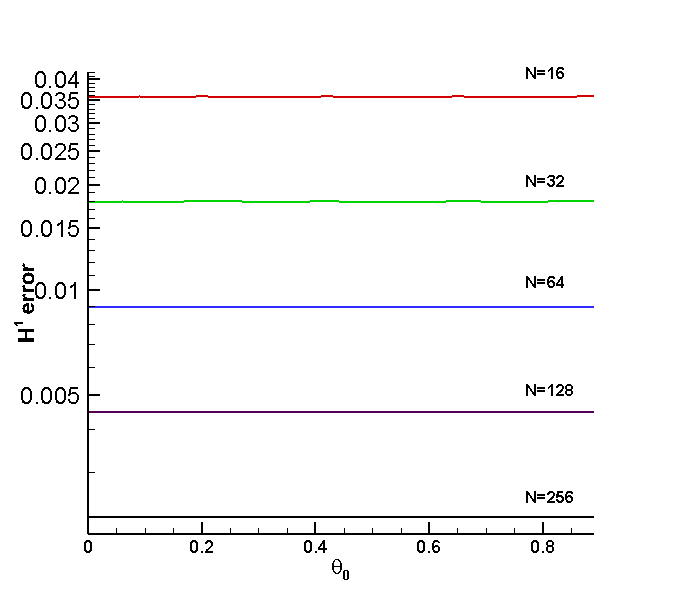}
\caption{Antisymmetric Nitsche CutFEM (\ref{CutAsym}) with $\gamma=1$, $\sigma=0.01$}
\end{subfigure}

\caption{Poisson-Dirichlet problem with CutFEM. The relative errors in $L^2(\Omega)$ and $H^1(\Omega)$ norms as 
functions of the rotation angle $\theta_0$, cf. Fig.~\ref{Dir:figres}.}
\label{Dir:cutres}
\end{figure}

As already mentioned in Introduction, our method (\ref{Dir:Ph})--(\ref{Dir:ah}) is very close to CutFEM methods, the essential difference being that we avoid the integration over $\Omega$, i.e. the numerical integration over the cut mesh elements. Although such an integration is in principle difficult to implement (in particular, it is more difficult than the integration over the segments of the boundary $\Gamma$), it is already available in FreeFEM, so that we can easily compare numerically the performance of CutFEM with that of our method.

We have considered the following variants of CutFEM with $V_h$ denoting everywhere the $P_1$ finite elements space on the mesh $\Th$, as in (\ref{Dir:Vh}). 

\begin{description}
\item[a)] A version with Lagrange multipliers approximated by $P_0$ finite elements on the cut triangles $\Th^\Gamma$ \cite{burman1}:\\
		Find $u_h \in V_h$, $\lambda_h \in W_h:= \{\mu_h \in L^2 (\Omega_h^\Gamma) : \mu_h |_T \in \mathbb{P}_0 (T)
  \  \forall T \in \mathcal{T}_h^\Gamma \}$ s.t.
\begin{align} \int_{{\Omega}} \nabla u_h \cdot \nabla v_h + \int_{\Gamma} \lambda_h v_h &=
   \int_{{\Omega}} fv_h &\quad &\forall v_h \in V_h 
	 \label{CutLam}\\
  \int_{\Gamma} \mu_h u_h - \sigma h \sum_{E \in \mathcal{E}_h^{\Gamma}}
   \int_{\Gamma} [\lambda_h] [\mu_h] &= \int_{\Gamma} g \mu_h &\quad &\forall
   \mu_h \in W_h
	\notag
\end{align}	
	\item[b)] A version based on the symmetric Nitsche method \cite{burman2}:\\
		Find $u_h \in V_h$ s.t.
\begin{align}
  \int_{\Omega} \nabla u_h \cdot \nabla v_h - &\int_{\Gamma}
  \frac{\partial u_h}{\partial n} v_h 
	- \int_{\Gamma} u_h \frac{\partial v_h}{\partial
  n} + \frac{\gamma}{h}  \int_{\Gamma} u_hv_h 
	+ {\sigma h \sum_{E \in\mathcal{F}_{\Gamma}} \int_E \left[ \frac{\partial u_h}{\partial n} \right] 
	  \left[ \frac{\partial v_h}{\partial n} \right]} &&
		\label{CutSym}\\
	= &\int_{\Omega_h} fv_h - \int_{\Gamma} g \frac{\partial
  v_h}{\partial n} + \frac{\gamma}{h}  \int_{\Gamma} gv_h \quad 
	&\quad&\forall v_h\in V_h
	\notag
\end{align}
	\item[c)]  A version based on the antisymmetric Nitsche method (similar to the above, only the sign in front of the third term changes) \cite{burmanghost}:\\
		Find $u_h \in V_h$ s.t.
\begin{align}
  \int_{\Omega} \nabla u_h \cdot \nabla v_h - &\int_{\Gamma}
  \frac{\partial u_h}{\partial n} v_h 
	+ \int_{\Gamma} u_h \frac{\partial v_h}{\partial
  n} + \frac{\gamma}{h}  \int_{\Gamma} u_hv_h 
	+ {\sigma h \sum_{E \in\mathcal{F}_{\Gamma}} \int_E \left[ \frac{\partial u_h}{\partial n} \right] 
	  \left[ \frac{\partial v_h}{\partial n} \right]} &&
		\label{CutAsym}\\
	= &\int_{\Omega_h} fv_h + \int_{\Gamma} g \frac{\partial
  v_h}{\partial n} + \frac{\gamma}{h}  \int_{\Gamma} gv_h \quad 
	&\quad&\forall v_h\in V_h
	\notag
\end{align}
\end{description}

The results are presented at Fig.~\ref{Dir:cutres}. We consider there the same setup as in Fig.~\ref{Dir:figres}, i.e. domain $\Omega$ given by (\ref{OmegaPhi}) rotated around the origin at a series of angles $\theta_0$, the exact solution given by (\ref{exact}). The stabilization parameters $\gamma$ and $\sigma$ are given in the captions (we have taken the same parameters for the antisymmetric version as for our method in Fig.~\ref{Dir:figres}, while a larger value for $\gamma$ was necessary for the symmetric version). Comparing the results in Figs. \ref{Dir:figres} and \ref{Dir:cutres}, we observe that all the methods have overall almost the same accuracy in the $H^1(\Omega)$ norm, but CutFEM produces better results (gaining by a factor around 2) when the error is measured in the $L^2(\Omega)$ norm. However, the performance of CutFEM (with the exception of the antisymmetric Nitsche variant) can drastically degrade at certain position of the domain $\Omega$ with respect to the mesh (the spikes of the error on the graphs as functions of the rotation angle $\theta_0$). This is certainly an implementation issue, which can be conceivably explained by inaccuracies in the numerical integration over the cut triangles and/or by an incorrect determination of such triangles due to the round-off errors. We do not attempt here to investigate this issue further, and merely note that the absence of integration over the cut triangles in our method (\ref{Dir:Ph})--(\ref{Dir:ah}) permits us to avoid some delicate implementation issues.

\subsection{Neumann boundary conditions}
We now turn to the Poisson-Neumann problem (\ref{Neu:P}). Our test case is similar to that used before: the domain is given by (\ref{OmegaPhi}) and the exact solution by (\ref{exact}). The right-hand side in (\ref{Neu:P}) is thus $f=0$ and the boundary condition $g$ is 
set up as $g=n\cdot\nabla u$ with the normal $n$ defined via the levelset function $\varphi$ in (\ref{OmegaPhi}) as $n=\frac{\nabla\varphi}{|\nabla\varphi|}$.  
We employ the method (\ref{Neu:Ph})--(\ref{Rob:ah}) taking the following parameter values: $\gamma_{div}=1$, $\gamma_{1}=10$, $\sigma=0.01$. The constraint $\int_\Omega u_h=0$ is enforced with the help of a Lagrange multiplier, thus increasing the size of the system matrix by 1. A convergence study under the mesh refinement is reported at Fig.~\ref{Neu:convh}. It confirms the optimal convergence order of the method in both $H^1$ and $L^2$ norms.

\begin{figure}[htp]
\centerline{
\includegraphics[width=0.4\textwidth]{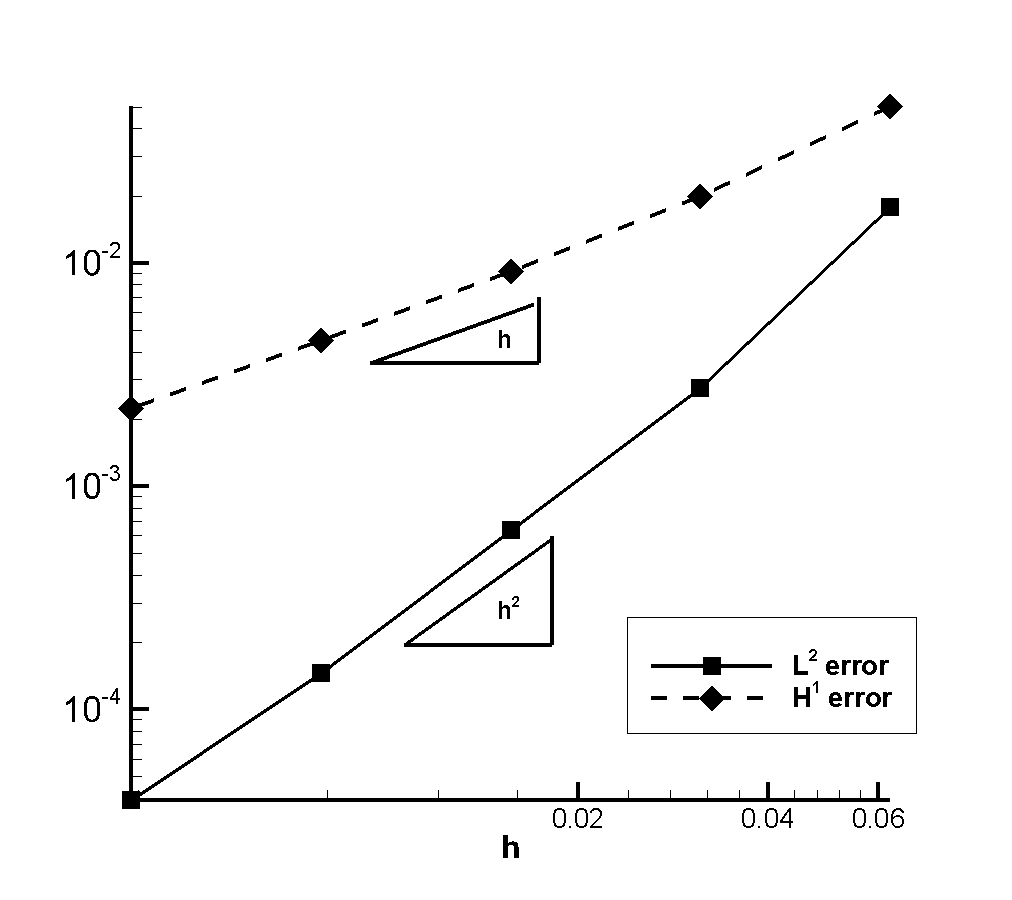}
}
\caption{Poisson-Neumann problem on domain and meshes as in Figs.~\ref{Dir:figdom}, \ref{Dir:figN16}: the relative errors in $L^2(\Omega)$ and $H^1(\Omega)$ norms as functions of $h$ under the mesh refinement.}
\label{Neu:convh}
\end{figure}

\begin{figure}[htp]
\centerline{
\includegraphics[width=0.4\textwidth]{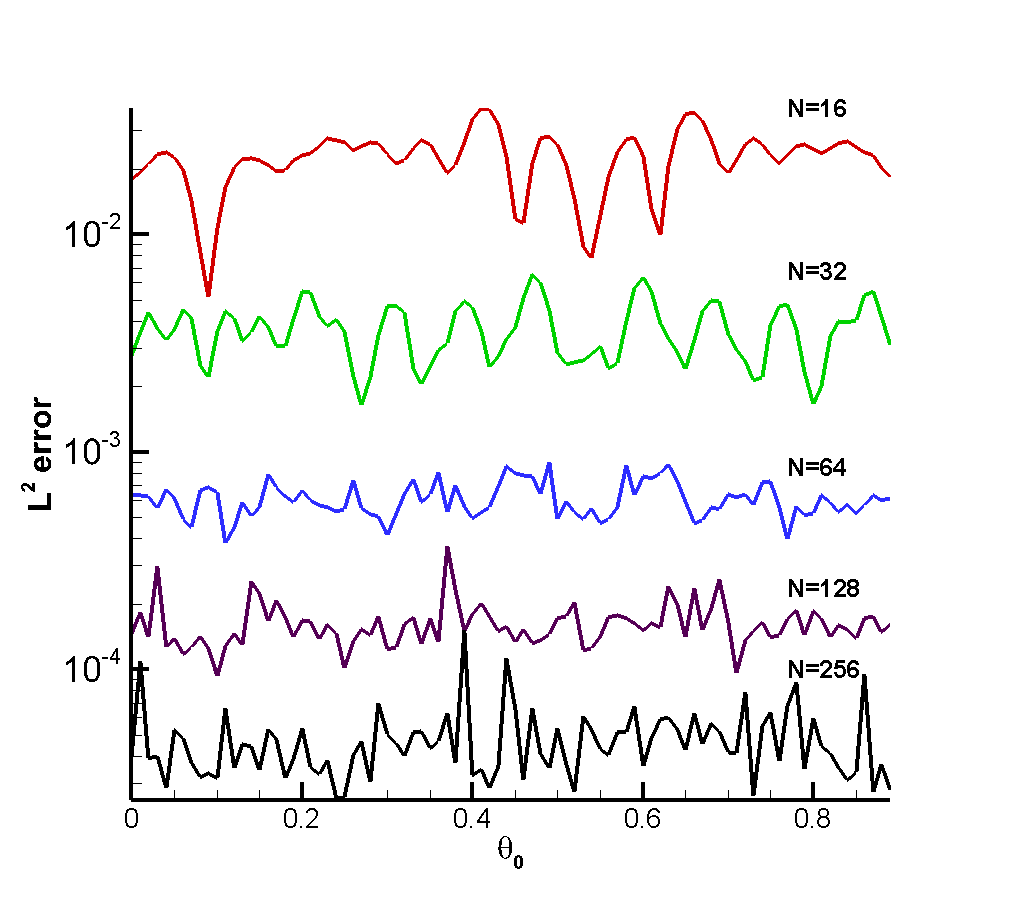}
\quad
\includegraphics[width=0.4\textwidth]{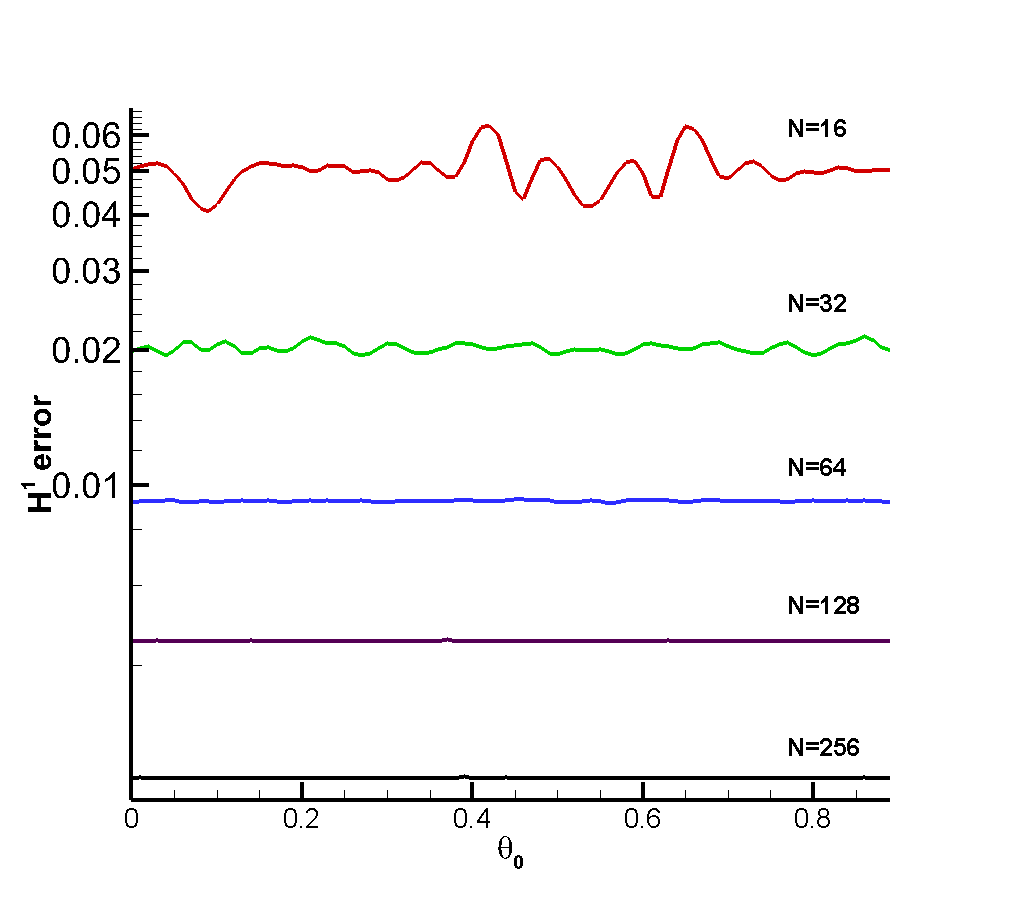}
}
\caption{Poisson-Neumann problem (\ref{Neu:P}) in the domain as in Fig. \ref{Dir:figdom} rotated by an angle $\theta_0$ around the origin: the relative errors in $L^2(\Omega)$ and $H^1(\Omega)$ norms as 
functions of the rotation angle $\theta_0$. The grad-div stabilization parameter set to $\gamma_{div}=1$. }
\label{Neu:figres}
\end{figure}

\begin{figure}[htp]
\centerline{
\includegraphics[width=0.4\textwidth]{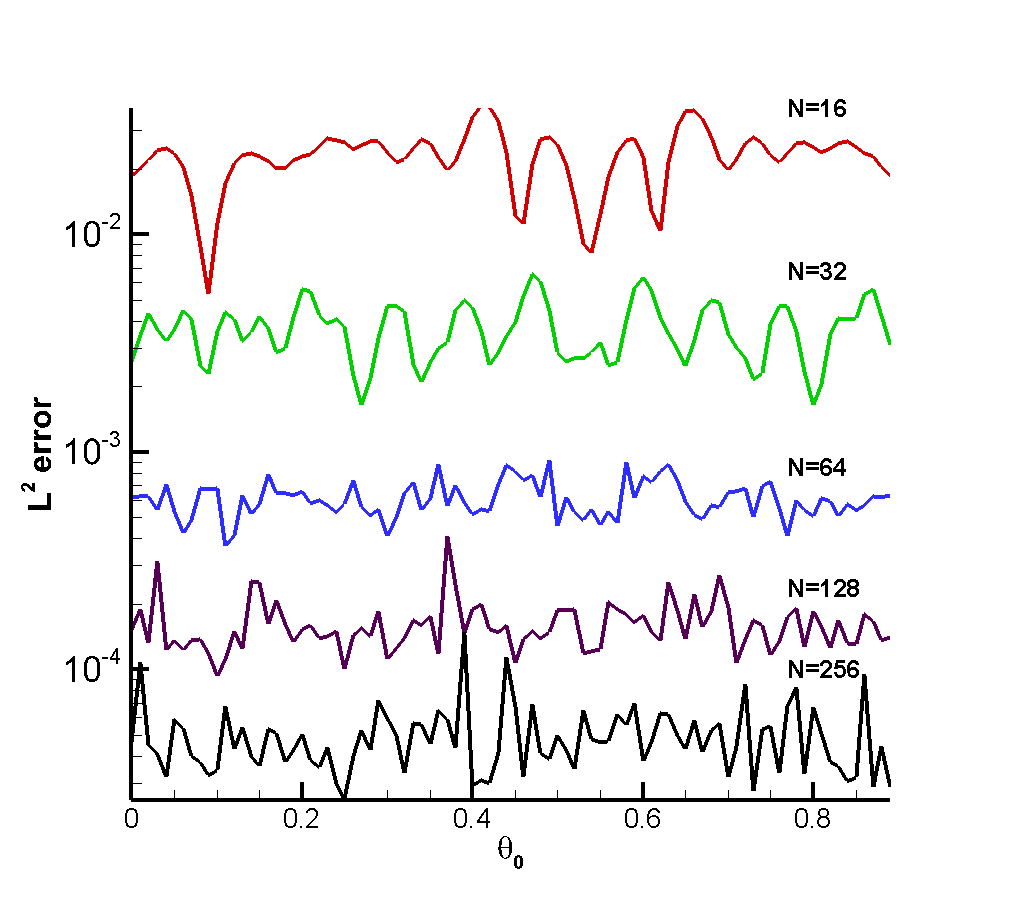}
\quad
\includegraphics[width=0.4\textwidth]{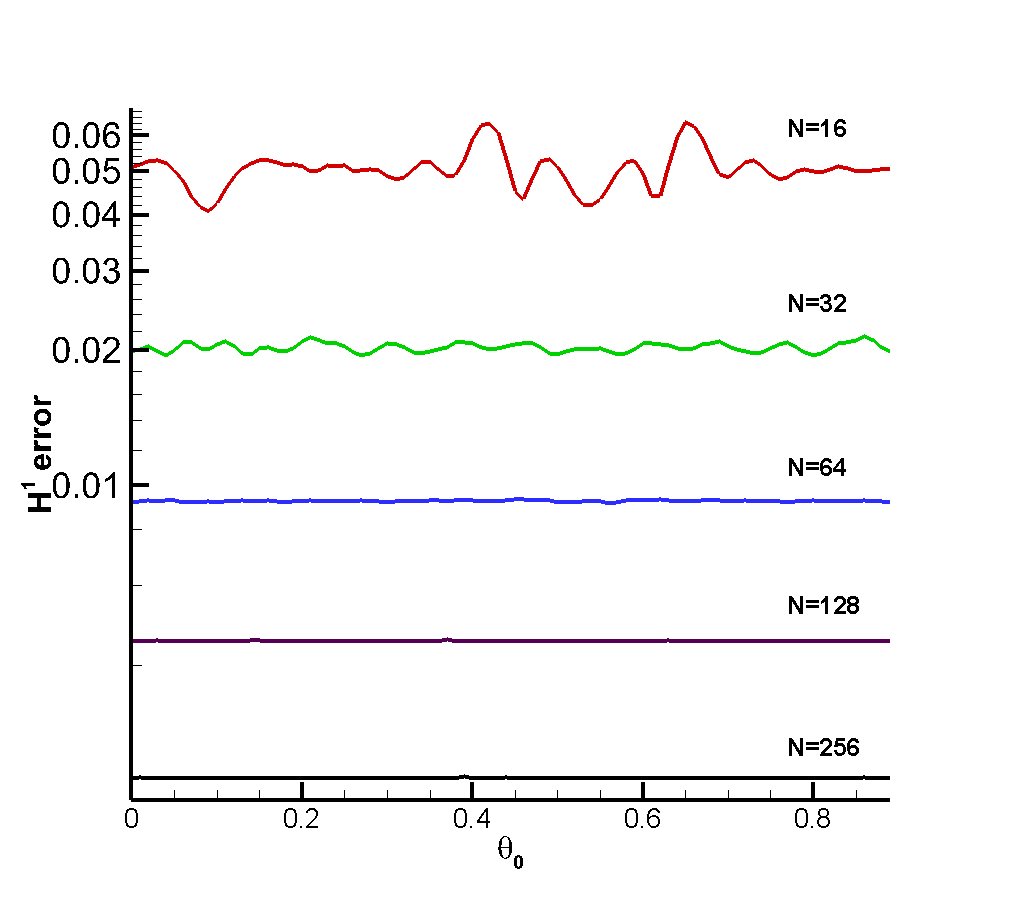}
}
\caption{Same setting as in Fig. \ref{Neu:figres} with the exception of the grad-div stabilization parameter which is chosen here as $\gamma_{div}=10h^2$. }
\label{Neu:figres_divh2}
\end{figure}

\begin{figure}[htp]
\centerline{
\includegraphics[width=0.4\textwidth]{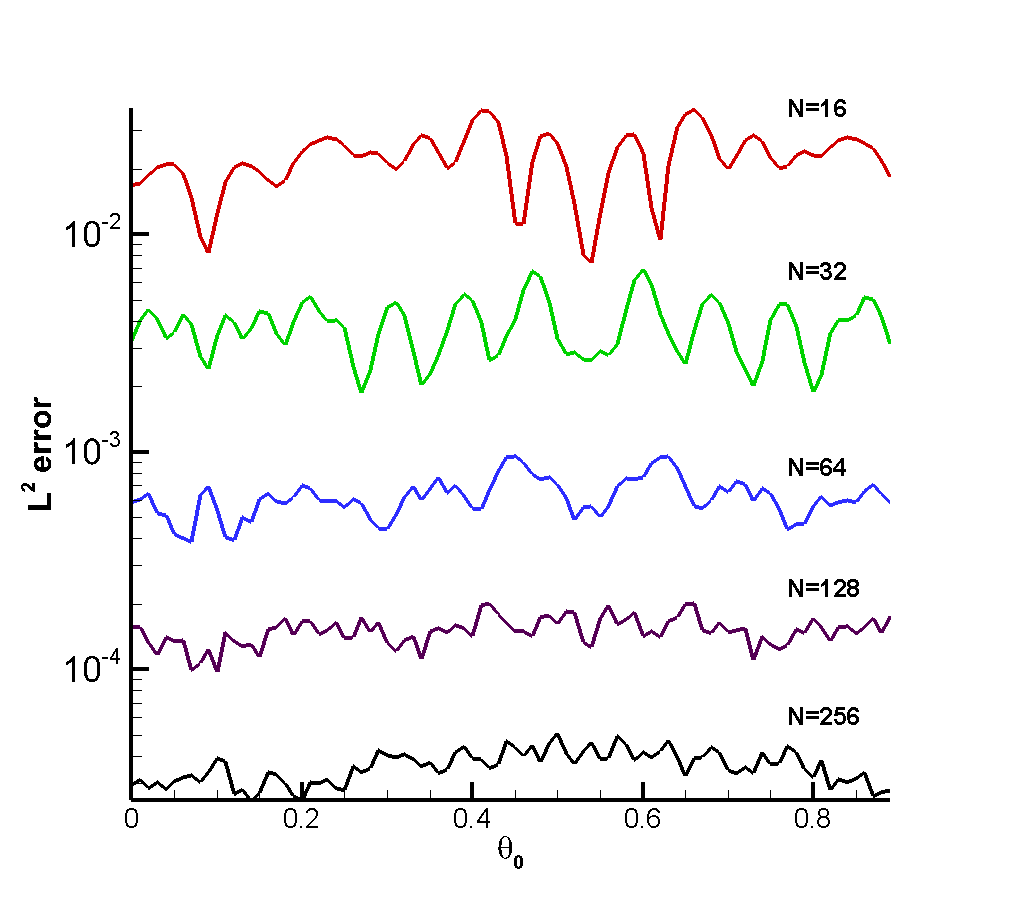}
\quad
\includegraphics[width=0.4\textwidth]{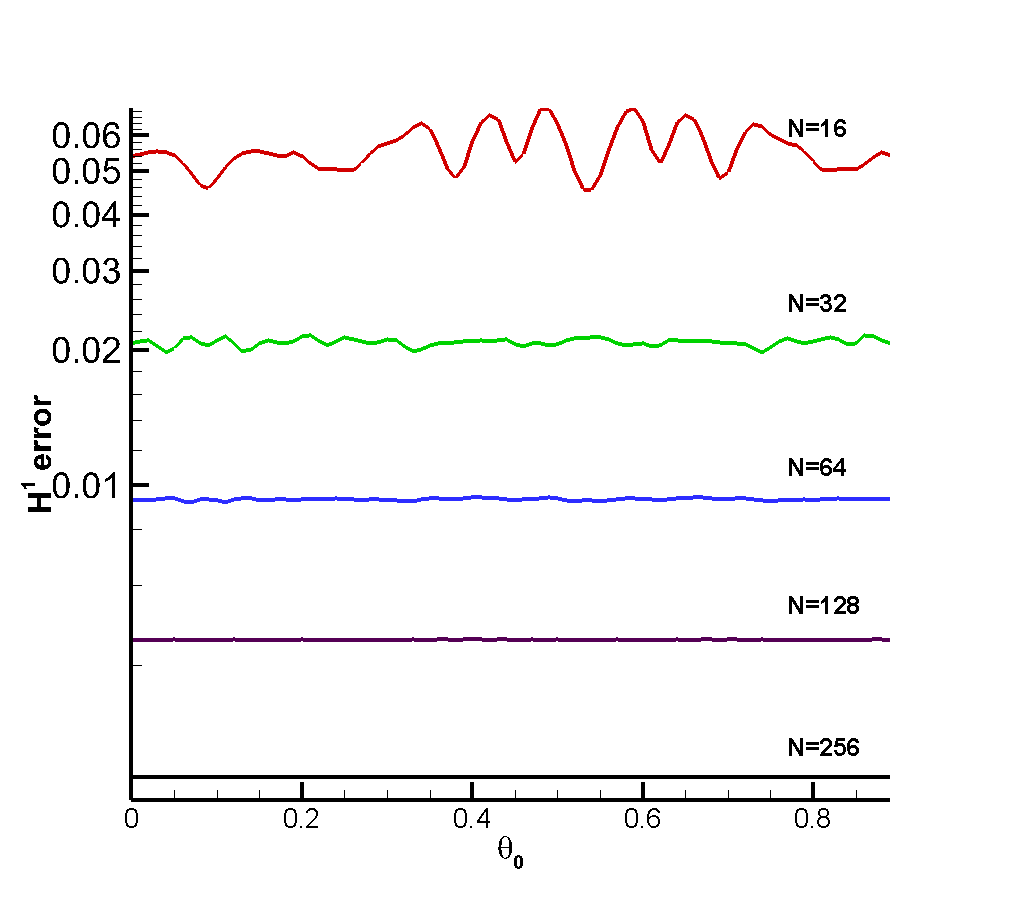}
}
\caption{CutFEM (\ref{CutNeu}) for Poisson-Neumann problem (\ref{Neu:P}) in the domain as in Fig. \ref{Dir:figdom}  rotated by an angle $\theta_0$ around the origin: the relative errors in $L^2(\Omega)$ and $H^1(\Omega)$ norms as 
functions of the rotation angle $\theta_0$.}
\label{Neu:cut}
\end{figure}

In order to explore the robustness of the method with respect to the placement of the physical domain on the computational mesh, we now redo the calculations above, rotating $\Omega$ by a series of angles $\theta_0$ ranging from $0$ to $\frac{2\pi}{7}$, same as in the Dirichlet case. Fig.~\ref{Neu:figres} presents the errors in $L^2(\Omega)$ and $H^1(\Omega)$ norms as functions of $\theta_0$ at different discretization levels. Overall, the errors are of the same order as in the Dirichlet case, but they are now much more sensitive to positioning the domain with respect to the mesh. This is especially true for the $L^2$ error whose variability does not fade out when the meshes are refined. 

\red{
The origin of this unfortunate phenomenon can be partially attributed to possible bad conditioning of the matrix of method (\ref{Neu:Ph})--(\ref{Rob:ah}). Indeed, the scaling of the grad-div stabilization term there is inconsistent with other terms. Loosely speaking, assuming that $y$ scales as $\frac{1}{h}u$,  the scaling of the term with $\gamma_{div}$ with respect to the mesh size is $1/h^2$ in 2D, while the scaling of all the other terms is 1. We have tried to recover the correct scaling of this term in numerical experiments reported in Fig.~\ref{Neu:figres_divh2}. The coefficient $\gamma_{div}$ is taken there proportional to $h^2$ rather than a mesh-independent constant. This scheme is thus not covered by the preceding theory but it gives essentially the same results in practice as the scheme with constant $\gamma_{div}$. In particular, the lack of convergence in the $L^2$ norm on certain geometrical configurations is still observed on the most refined meshes.}

\red{
It is also interesting to compare the results produced by the scheme (\ref{Neu:Ph})--(\ref{Rob:ah}) with those of CutFEM:
Find $u_h \in \tilde{V}_h$ s.t.
\begin{equation}
\label{CutNeu}
  \int_{\Omega} \nabla u_h \cdot \nabla v_h 
	+  {\sigma h \sum_{E \in\mathcal{F}_{\Gamma}} \int_E \left[ \frac{\partial u_h}{\partial n} \right] 
	  \left[ \frac{\partial v_h}{\partial n} \right]} 		
	= \int_{\Omega_h} fv_h + \int_{\Gamma} gv_h \quad 
	\qquad\forall v_h\in \tilde{V}_h
\end{equation}
The results are reported in Fig.~\ref{Neu:cut} taking the stabilization parameter $\sigma=0.01$. The error curves in the $H^1(\Omega)$ norm are practically the same as those for (\ref{Neu:Ph})--(\ref{Rob:ah}). We also observe once again a big variability of the  $L^2(\Omega)$ with respect to the geometry. However, these variations seem to fade out under the mesh refinement, contrary to the scheme (\ref{Neu:Ph})--(\ref{Rob:ah}). 
}

Extensive numerical experiments on the influence of the stabilization parameters $\gamma_{div}$, $\gamma_1$ and $\sigma$ in (\ref{Rob:ah}) have been also conducted. Similar to the Dirichlet case, the accuracy does not seem to deteriorate catastrophically in the limit $\gamma_{div},\sigma\to 0$ and the method is not too much sensitive to the parameters in a wide range. We choose not to go into further details since the behaviour with respect to 3 parameters is difficult to summarize efficiently in figures or tables.

\color{black}
\subsection{A test case in 3D}
In a slight variation to the preceding numerical experiments and to the theoretical analysis of the paper, we present here the results for the diffusion equation in a 3D domain:
\begin{equation}
  - \Div(D\nabla u) = f \text{ in } \Omega, \qquad u = 0 \text{ on } \Gamma
  \label{diff}
\end{equation}
with a non-constant diffusion coefficient $D$.
We take $\Omega\subset\RR^3$ as the ball of radius 1 centred at the  origin. Denoting $r$ the distance to the origin, we choose the diffusion coefficient and the exact solution as
$$
D=\frac{1}{1+r}, \quad u=\cos\left(\frac{\pi}{2}r\right)
$$
and adjust the right-hand side $f$ accordingly.
The method (\ref{Dir:Ph})--(\ref{Dir:ah}) is modified as follows: Find $u_h \in V_h$ (the $P_1$ continuous finite elements on mesh $\Th$) such that
\begin{equation*}
  \int_{\Omega_h} D\nabla u_h \cdot \nabla v_h - \int_{\Gamma_h}
  D\frac{\partial u_h}{\partial n} v_h + \int_{\Gamma} u_h D\frac{\partial v_h}{\partial
  n} + \frac{\gamma}{h}  \int_{\Gamma} u_hv_h + \sigma h \sum_{E \in
  \mathcal{F}_{\Gamma}} \int_E \left[ \frac{\partial u_h}{\partial n} \right]
  \left[ \frac{\partial v_h}{\partial n} \right] 
	= \int_{\Omega_h} fv_h \quad \forall v_h  \in V_h 
\end{equation*}

The numerical experiments have been done in FreeFEM using the HYPRE linear solvers via PETSc. The computational meshes $\Th$ were obtained from the background uniform tetrahedral meshes on a sufficiently large cube by getting rid of the tetrahedra outside $\Omega$, cf. some example in Fig.~\ref{3Dmesh}. The following stabilization parameters were used: $\gamma=1$, $\sigma=0$. A mesh convergence study is reported in Fig.~\ref{3Dconvh} and confirms the optimal convergence rates.

\begin{figure}[htp]
\centerline{
\includegraphics[width=0.4\textwidth]{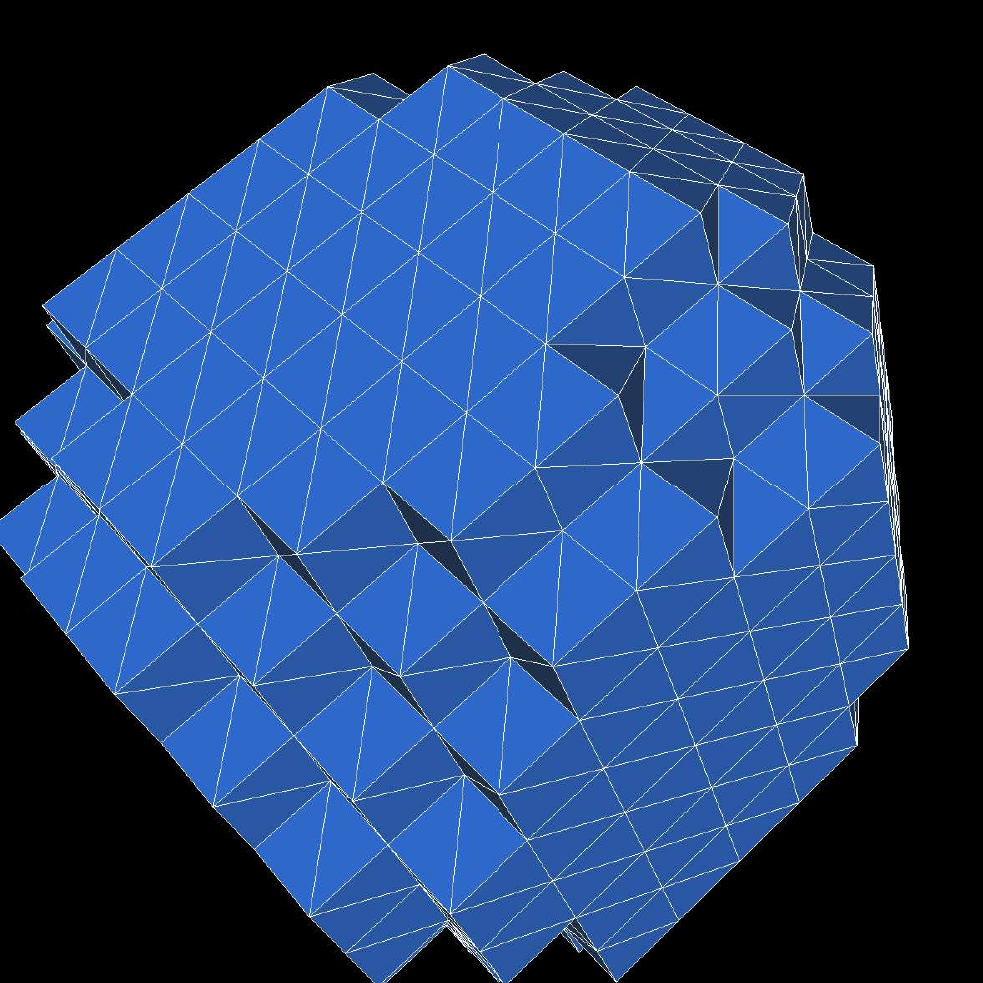}
\quad
\includegraphics[width=0.4\textwidth]{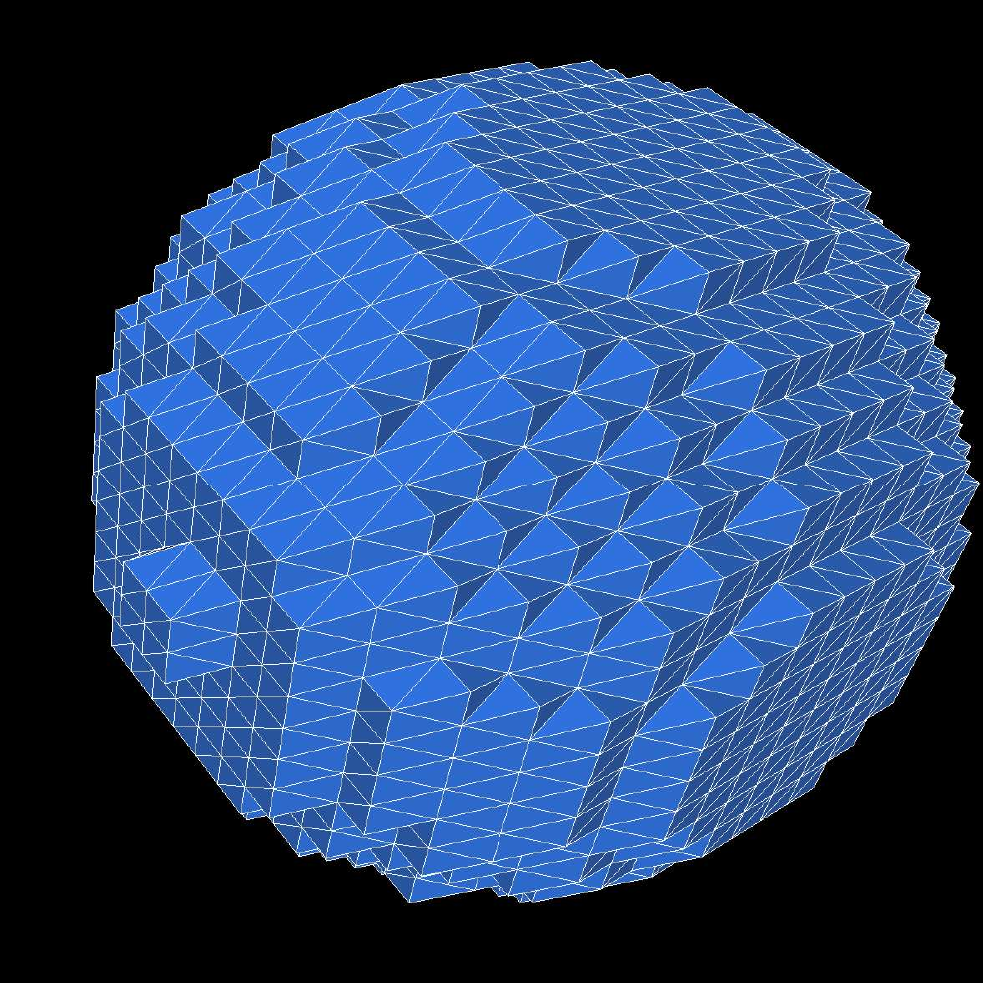}
}
\caption{The computational meshes used in 3D numerical experiments. Left: $\Th$ obtained from the $10\times 10\times 10$ background mesh; right: $\Th$ obtained from the $20\times 20\times 20$ background mesh.}
\label{3Dmesh}
\end{figure}

\begin{figure}[htp]
\centerline{
\includegraphics[width=0.4\textwidth]{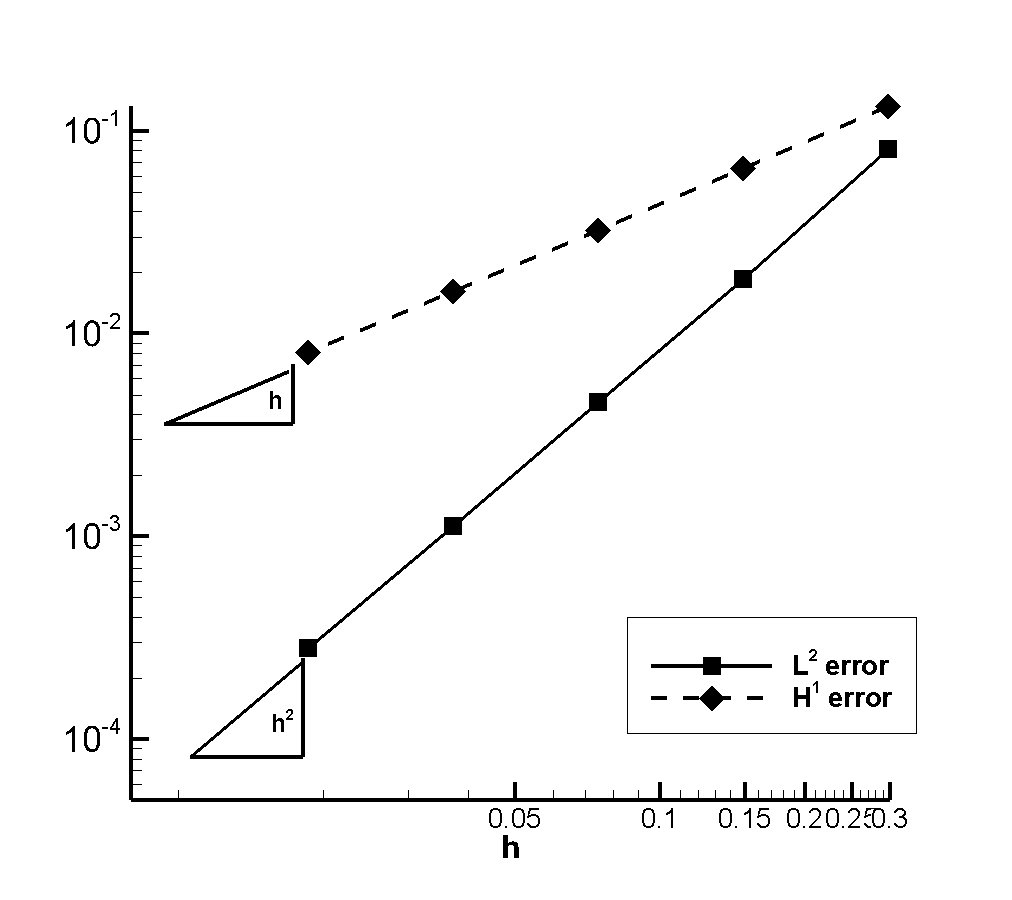}
}
\caption{Diffusion problem (\ref{diff}) on the unit ball and meshes as in Fig.~\ref{3Dmesh}: the relative errors in $L^2(\Omega)$ and $H^1(\Omega)$ norms as functions of $h$ under the mesh refinement.}
\label{3Dconvh}
\end{figure}
\color{black}

\section{Conclusions}
We have presented an optimally convergent method of the fictitious domain/XFEM/CutFEM type avoiding the numerical integration on cut mesh elements. The numerical experiments confirm the optimal convergence  and the robustness of the method (with a possible slight deficiency on the level of $L^2$ errors in the case of Neumann boundary conditions). We have restricted ourselves to simple model problem (Poisson equation, easily generalisable to the diffusion equation) in this first publication, but the methods should be applicable in more realistic  settings as, for example, elasticity problem on a cracked domain (the method would then be similar to XFEM of \cite{moes99} but avoiding the integration on the mesh elements cut by the crack).
Further research could be aimed, apart from extensions to more complex problems, at a finer understanding  of the influence of stabilization parameters.

\section*{Acknowledgements}
I am grateful to Gr\'egoire Allaire for stimulating discussions and suggestions concerning the implementation. The implementation in FreFEM++ would not be possible without advices and help generously provided by  Fr\'ed\'eric Hecht and Pierre Jolivet. \red{I have also greatly appreciated the insightful and thought provoking comments of one of the anonymous reviewers.}



\bibliographystyle{elsarticle-num-names} 
\bibliography{cutfem}

\end{document}